\documentclass{au}
 \presentedby{\dots}
 \received{\dots}{\dots}

\usepackage{multirow,hhline}
\usepackage{amssymb,amsmath,amsfonts,amsthm,url}
\usepackage[mathscr]{eucal}
\usepackage{bm}

\usepackage{pgf}
\usepackage{tikz}

\usetikzlibrary{positioning,chains,decorations.pathmorphing,matrix,arrows,shapes.geometric}

\numberwithin{equation}{section} 

\theoremstyle{plain}
\newtheorem{theorem}{Theorem}[section]

\newtheorem{lemma}[theorem]{Lemma}
\newtheorem{proposition}[theorem]{Proposition}
\newtheorem{corollary}[theorem]{Corollary}

\theoremstyle{definition}

\newcommand{\class}[1]{\boldsymbol{\mathfrak{#1}}}
\newcommand{\cat}[1]{\boldsymbol{\mathscr{#1}}}
\newcommand{\str}[1]{\mathbf{#1}}
\newcommand{\fnt}[1]{\mathsf{#1}}
\newcommand{\ope}[1]{\mathbb{#1}}
\newcommand{\defn}[1]{{\emph{#1}}}

\newcommand{\CA}{\cat A}
\newcommand{\CX}{\cat X}
\newcommand{\CP}{\cat P}
\newcommand{\CZ}{\cat Z}
\newcommand{\CCD}{\cat D}
\newcommand{\LMn}{\cat{LM}_n}
\newcommand{\Mn}{\cat{M}_n}
\newcommand{\CM}{\class M}
\newcommand{\CN}{\class N}  
\newcommand{\CMT}{\twiddle {\mathfrak M}}
\newcommand{\X}{\str X}
\newcommand{\Y}{\str Y}
\newcommand{\Z}{\str Z}
\newcommand{\M}{\str M}
\newcommand{\A}{\str A}
\newcommand{\B}{\str B}
\newcommand{\C}{\str C}
\newcommand{\Free}{\str{F}\!} 
\newcommand{\fntH}{\fnt{H}}
\newcommand{\Uca}{\fnt{U}_{\CA}}
\newcommand{\two}{\boldsymbol 2}
\newcommand{\twoT}{\two_{\Tp}}
\newcommand{\Cong}{\text{Con}}
\newcommand{\Tp}{\mathscr{T}}
\newcommand{\ISP}{\ope{ISP}}

\newcommand{\DP}{${\CCD\text{-}\CP\,}$}

\renewcommand{\leq}{\leqslant}
\renewcommand{\geq}{\geqslant} 
\renewcommand{\preceq}{\preccurlyeq} 
\newcommand{\w}{\omega}
\newcommand{\chisub}[1]%
{\hbox{\raise.4ex\hbox{$\chi_{\hbox{}_{\scriptstyle 
#1}}$}}\!}

\newcommand{\twiddle}[1]{{\smash{\underset{\raise.375ex\hbox{$\smash\sim$}}
       {\mathbf{#1}}}\vphantom{\underline{\mathbf{#1}}}}} 
\newcommand{\stwiddle}[1]{\smash{\underset{\smash{\raise.1ex\hbox{\small$\sim$}}}
                         {\mathbf{#1}}}\vphantom{#1}}
 \DeclareMathOperator{\IScP}{{\mathbb{IS} _{\mathrm{c}}%
 \mathbb{P}^+}}
\newcommand{\MT}{\twiddle M}

\newcommand{\dotbigcup}{\overset{\raise.6ex\hbox{$\smash\cdot$}\,}{\smash\bigcup\,}}

\newcommand{\eysy}{\hbox{$\text{\rm E}^{\text{\checkmark}}\text{\rm S}
^{\text{\checkmark}}$}}
\newcommand{\eysn}{\hbox{$\text{\rm E}^{\checkmark}\text{\rm S}^{\times}$}}
\newcommand{\ensy}{\hbox{$\text{\rm E}^{\times}\text{\rm S}^{\checkmark}$}}
\newcommand{\ensn}{\hbox{$\text{\rm E}^\times\text{\rm S}^\times$}}
 
\newenvironment{newlist}
   {\begin{list}{}{\setlength{\labelsep}{0.25cm}
                   \setlength{\labelwidth}{0.65cm}
                      \setlength{\leftmargin}{0.9cm}}}
   {\end{list}}

\newcommand{\bigdu}{\smash{\overset{\smash{\lower 1ex\hbox{$\cdot$}}}{\bigcup}\, }}

\newcommand{\powerset}[1]{{\raise.5ex\hbox{$\wp$}#1}} 

\newcommand{\dotbigcupdisp}{\mbox{\Large{$\textstyle\overset{\raise.65ex\hbox{$\smash\cdot$}}{\smash\bigcup}$}}}

\begin{document}
\title{Coproducts of distributive lattice-based algebras}
\author{L. M. Cabrer}
\email{lmcabrer@yahoo.com.ar}
\address{Mathematical Institute, 
University of Oxford\\
Radcliffe Observatory Quarter\\Oxford OX2 6GG \\UK}

\author{H. A. Priestley}
\email{hap@maths.ox.ac.uk}
\address{Mathematical Institute,
University of Oxford\\
Radcliffe Observatory Quarter \\
Oxford OX2 6GG \\UK}

%%%%%%%%%%%%%%%%%%%%%%%%%%%%%%%%%%%%%%%%%

\begin{abstract}
This paper presents a systematic study of coproducts.  This is carried out 
principally, but not exclusively, for
finitely generated quasivarieties $\CA$ that admit a  (term) reduct in the variety $\CCD$ of bounded distributive 
lattices.   In this setting 
we present necessary and sufficient conditions on~$\CA$ for the forgetful functor $\fnt U_{\CA}$
from~$\CA$ to~$\CCD$
 to preserve coproducts. 
 We  also investigate the possible behaviours 
 of~$\fnt U_{\CA}$
as regards coproducts in~$\CA$ under weaker assumptions. 
Depending on the properties exhibited by the functor, different procedures 
are then available for describing these coproducts.
We classify a selection of well-known varieties within our scheme,
thereby unifying earlier results and obtaining some new ones.  

The paper's methodology draws heavily on duality theory.   
We use Priestley duality as a tool and our 
descriptions of coproducts are given in terms of  this duality.
We also  exploit natural duality theory, specifically
multisorted piggyback dualities, in our  analysis of  the behaviour of 
the 
forgetful functor into $\CCD$.
 In the opposite direction,  we reveal that
the type of natural duality that the class $\CA$ can possess is governed by properties of coproducts in $\CA$ and the way in which the classes $\CA$ and 
$\fnt U_{\CA}(\CA)$  interact.
\end{abstract}

\maketitle

%%%%%%%%%%%%%%%%%%%%%%%%%%%%%%%%%%%%%%%%%%%%%%%%%%%%%%%%%%%%%%%%%%%%%%%%%%%%

\section{Introduction} \label{Sec_Intro}  
Coproducts in varieties of algebras have been extensively studied and 
a multitude of papers related to  this topic can be found in the literature: 
\cite{Qu72,CF77,CF79,Ci79,BeGr,Pa,DM}, and more. 
Often
the characterisation of coproducts has been motivated by  interest in  
amalgamation properties, determination of free objects,
axiomatisations, or 
colimits more widely.
Whatever the motivation, the analysis of coproducts has generally been variety-specific, 
relying  on  tools tailored
 to particular classes of algebras.  A recurring theme, however, is 
the use of a categorical duality.
The main objective of this paper is to  give 
 a  uniform
 treatment, based on duality theory,
of coproducts in classes of algebras that admit a bounded distributive lattice reduct.  Within this setting,  we contribute new theoretical results and 
thereby provide a  unified perspective on characterisations (most already known, a few new) of coproducts  in particular classes.

Let $\CCD$ denote the category  whose objects are  bounded distributive lattices and whose morphisms are the lattice homomorphisms which preserve the bounds.
We will say that a class of algebras $\CA$ with language~$\mathcal{L}$ is a \defn{$\CCD$-based class} if the algebras of 
$\CA$ have
reducts in~$\CCD$.
Specifically,  this is the case if
there exists a subset $\{\wedge,\vee,0,1\}$ of $\mathcal{L}$ such that, for each algebra $\A\in \CA$,  the reduct $(A,\wedge^{\A},\vee^{\A},1^{\A},0^{\A})$ is in~$\CCD$. 
 (More generally, it suffices to assume that the class $\CA$ is term-equivalent to a class which has this property.) 
Many familiar classes $\CA$  of algebras arising as models
for propositional logics, for example, are $\CCD$-based.

 We recall that 
 it is well known that coproducts  in $\CCD$ have  a simple description via Priestley duality (see Section~\ref{Sec_Prel}). 
 In this paper we shall study coproducts in  $\CCD$-based categories  of algebras by relating them to  coproducts in $\CCD$.
To that end, we first observe that if $\CA$ is a $\CCD$-based class of algebras then
 the assignment $\A\mapsto (A,\wedge^{\A},\vee^{\A},0^{\A},1^{\A})$ and $h\mapsto h$ determines a (forgetful) functor~$\Uca$  from~$\CA$ into~$\CCD$.
We will concentrate on the problem of determining 
when $\Uca$ 
{preserves coproducts}. 
We give the precise (categorical) definition of this property  in Section~\ref{Sec_Cop} (see also \cite[Section V.4]{McL1969}).  We note here that
 it implies, but is stronger than, the 
 requirement that 
$\Uca(\coprod_{\CA}(\class{K}))\cong\coprod_{\CCD}\Uca(\class{K})$, 
 for each set $\class{K}\subseteq \CA$.
Henceforth, we shall omit the subscript from the forgetful functor~$\fnt U$  when working with some fixed class  $\CA$.

We shall, generally but not exclusively, restrict attention to the situation in which 
$\CA$  is a \defn{finitely generated quasivariety}, that is, there exists a finite set of finite algebras
$\CM \subseteq \CA$ such that $\CA = \ISP(\CM)$.  
Such a set $\CM$ exists whenever $\CA$ is a $\CCD$-based variety 
expressible as $ \ope{HSP}(\CN)$, where $\CN$ is a finite set of finite algebras.
   This is a consequence of J\'onsson's Lemma:  we have  
$\CA = \ISP(\CM)$, where  $\CM\subseteq \ope{HS}(\CN)$ 
 (see \cite[Corollaries~6.9 and~6.10]{BuSa}).  
Working with quasivarieties 
certainly  
ensures the existence of coproducts for each set of algebras in the class.

 The principal result of this paper is Theorem \ref{MainTheo} (Coproduct Preservation Theorem).   Given $\CA$ and $\fnt U$ as above 
we  show that, for any set $\class{K}$ in $\CA$,  there is a natural  $\CCD$-homomorphism
$
 \chisub{\class{K}}\colon \coprod\fnt U(\class{K})\to \fnt U(\coprod\class{K})$ (see Fig.\ref{fig:coprod}(b)). 
  Theorem~\ref{MainTheo} gives necessary
and sufficient conditions for $\chisub{\class{K}}$ to be an isomorphism
for any set $\class{K}$.  
It is obtained 
by combining two results, Theorems~\ref{Theo:EmbeddingCoprod} and~\ref{Theo:OntoCop},  the first giving conditions for each map  $\chisub{\class{K}}$
to be surjective  (property (S))    and the second conditions for  each map $\chisub{\class{K}}$  to be an embedding   (property (E)). 
Theorem~\ref{MainTheo} then gives two sets of conditions, each of which is equivalent to the satisfaction of both (E) and (S).
The first set  can be viewed  as  specifying a particular, and very special,    
form of interaction  between $\CA$ and the subclass $\fnt U (\CA)$ of
$\CCD$. 
There is required   
to be an  algebra $\M$ such that 
$\CA = \ISP(\M)$ and  a $\CCD$-homomorphism~$\w$ from~$\fnt U (\M)$ into the two-element algebra $\two $ in $\CCD$ 
with two properties, which we now introduce. The first
 is a separation condition, $\text{(Sep)}_{\M,\w}$:
for $a \ne b$ in $\M$, there is $u \in \CA(\M,\M)$ such that 
$\w(u(a)) \ne \w(u(b))$.  The second property  demands  
that
for each algebra $\A\in\CA$ and each sublattice $\str{L}$ of $\fnt U(\A)$ there exists a subalgebra of $\A$ maximal for the inclusion order among the subalgebras of $\A$ contained in $\str{L}$. 
Interestingly, this property has a syntactic counterpart, 
to the effect that an
 arbitrary term in the language~$\mathcal{L}$ of $\CA$ is  equivalent in $\CA$ to  a term built from unary terms in  the language~$\mathcal{L}$
and terms in the language of $\CCD$ 
(see Theorems~\ref{Theo:EmbeddingCoprod} and~\ref{Theo:GeneralS}).
  The second, alternative, set of
conditions in Theorem~\ref{MainTheo} is included principally to allow us to demonstrate that whether or not coproducts are preserved is a decidable problem; see the discussion following the theorem.

Theorems~\ref{Theo:EmbeddingCoprod} and~\ref{Theo:OntoCop} are
 of independent interest.
They allow us  
 to  analyse  the behaviour of~$\fnt U$ under weaker conditions than the ones presented in the preceding paragraph (see the flowchart  in Fig.~\ref{flowchart}).   
Therefore 
even
when the functor $\fnt U$ fails to 
commute with coproducts we may 
be able to 
describe 
coproducts.
Table~\ref{Tab:ProcedureByType}   summarises 
the strategies we  have available for 
doing this.

In Section~\ref{Sec_Prel} we set out
the duality theory on which we shall rely.  In particular  we set up, briefly and in a self-contained way,  the framework 
for the Multisorted Piggyback Duality Theorem (stated as Theorem~\ref{genpig}).   
We conclude Section~\ref{Sec_Prel}  with Theorem~\ref{Theo:RevEng}.  This allows us to relate
the dual structures supplied by Theorem~\ref{genpig} for the algebras in a 
finitely generated $\CCD$-based  quasivariety to the Priestley dual spaces of 
their $\CCD$-reducts. 
 We then have in place  the machinery we need to 
investigate coproducts via duality theory. 
 Theorem~\ref{Theo:RevEng} 
is new, though what it tells us about
dualities for particular quasivarieties was already known in certain 
cases.  
Section~\ref{Sec_Cop}---the core
of the paper---presents the statements and proofs of our main 
results, as outlined above.  
Anyone  conversant with the piggybacking method \cite{DWpig,DP87,CD98} will have recognised our first set of  conditions for preservation of coproducts as statements relating to  piggyback dualities. 
Indeed, the classification of classes $\CA$ according to whether property  (S) and/or property (E) holds or fails governs the form a  piggyback duality for $\CA$
will take.   Moreover, via 
property (E),  we are able  easily to  prove that 
$\CA$ admits free products 
if and only if  $\CA$ is  generated as a  quasivariety by a 
single finite algebra (this specialises a result known to hold
more generally).
Details are given in Theorem~\ref{Theo:FreeProd} and Theorem~\ref{Theo:CoproToNatDual}.

 In Section~\ref{Sec:CopSub} we venture beyond the confines 
of finitely generated $\CCD$-based quasivarieties, with the aim of 
revealing how far certain results in Section~\ref{Sec_Cop} hold in
greater generality.
 Finally, in Section~\ref{Sec_Ex},  we apply 
our results and techniques 
to particular well-known classes of finitely generated varieties.   Primarily our catalogue unifies pre-existing descriptions  of coproducts, both as regards the descriptions 
 and as regards the methodology 
for obtaining them. 

A few comments should be made on our  assumption 
that the quasivarieties we consider  be finitely generated. 
  This comes into play 
 to 
guarantee that we can set up the dualities we shall require (see Section~\ref{Sec_Prel}).  
In self-justification, we note 
that many interesting quasivarieties are 
finitely generated.  Moreover,   in any class $\CA$ of algebras which is locally finite
 (that is, finitely generated algebras are finite),  
any  quasivariety generated by a  finitely generated free object in $\CA$ is finitely generated. This simple observation demonstrates that our results
 have interesting consequences 
beyond the finitely generated setting (see Section~\ref{Sec:CopSub}).

\section{Preliminaries: duality theories}\label{Sec_Prel}

Take, as above, $\CA$ to be 
a $\CCD$-based and finitely generated quasivariety.
In this section we outline the results  on which our analysis of coproducts in~$\CA$ will rest.   
Underlying our strategy throughout will be duality theory, in two forms.  
Our main results are obtained by operating with these two forms in tandem,  and toggling between them.  
First we briefly discuss the role played 
 by Priestley duality 
as a platform on top of which dualities for classes of $\CCD$-based 
algebras can be built.    For many such classes, 
this technique gives a  valuable, set-based,  representation theory.  However,  although 
isolated results exist in the literature  for particular classes,  such representations
do not lend themselves well in general to the description of coproducts. 
Secondly, we venture a little way into the theory of  (multisorted) natural dualities as it applies to finitely generated $\CCD$-based quasivarieties, 
since such dualities, in common with Priestley duality itself, have good categorical properties.  The key theorem on which we shall rely
is  the Multisorted Piggyback Duality Theorem.
We 
present the bare minimum of the theory  
necessary to state it (Theorem~\ref{genpig}).  It allows us immediate access to coproducts,  via dual structures which are (multisorted) cartesian products.   
This would be of
little assistance  were it not possible  directly 
 to retrieve the coproduct, or at least the  Priestley dual space of its $\CCD$-reduct,  from the  dual structure. 
Theorem \ref{Theo:RevEng} provides exactly the  translation tool we need. 
(We remark that the usefulness  of Theorem~\ref{Theo:RevEng} potentially extends
 well
  beyond the applications to coproducts  given in this paper.)

Priestley duality for (bounded) distributive lattices establishes 
 a dual equivalence between the category~$\CCD$ and the category~$\CP$ of Priestley spaces, that is, compact totally order-disconnected spaces with  continuous order-preserving maps between such spaces as the morphisms. 
We shall below assume familiarity with Priestley duality (summaries can be found in,
for example, \cite[Chapter~11]{ILO2} and \cite{GLTDP}) but do
indicate here how 
   the duality  is set within its rightful categorical framework.
We  recall that $\CCD = \ISP(\two)$, 
 the class of isomorphic copies of subalgebras of products of 
$\two$;
and 
$\CP = \IScP(\twoT)$, the class of isomorphic copies of closed substructures of powers of $\twoT$, where $\twoT$ denotes the $2$-element   
Priestley space 
$(\{0,1\}, \leq, \Tp)$, in which $0 < 1$ and $\Tp$ is the discrete topology.
The duality between $\CCD$ and~$\CP$ is set up by the natural hom-functors, $\fnt H = \CCD(-,\two)$ and $\fnt K = 
\CP (-,\twoT)$. The accounts in  
\cite{GLTDP} and, in    
 \cite[Chapter~5]{CD98} highlight
the highly amenable properties this duality possesses, and why.  
We note two key facts.
Firstly, 
products in the category~$\CP$ are concrete (that is, cartesian) products.  This implies that~$\fnt H$ maps coproducts to cartesian products. 
Secondly, in the terminology of \cite{CD98}, the duality is strong. 
Consequently,  
for a $\CCD$-morphism~$f$, the dual map 
 $\fnt H(f)$
is surjective if and only if~$f$ is injective and that
$\fnt H(f)$ is an embedding in~$\CP$ if and only if~$f$
 is surjective.

Now assume that $\CA$ is a variety or quasivariety of $\CCD$-based 
algebras and assume that  we wish to find a category $\CZ$ of structured spaces 
we can use to represent the algebras and morphisms of $\CA$ in terms of~$\CZ$.  One  way  to try to proceed 
is first to take the class $\fnt U(\CA)$ of $\CCD$-reducts of members of 
$\CA$ and  to equip the associated class $\fnt{HU}(\CA)$
of Priestley spaces 
with additional (relational or functional) structure to enable each 
$\fnt{KHU}(\A)$ to be made into an algebra in~$\CA$
isomorphic to $\A$; this gives the objects of $\CZ$.  Then an associated  
subclass of morphisms needs to be identified to serve as the 
$\CZ$-morphisms.  
When an equivalence between $\CA$ and $\CZ$ is 
constructed in this
manner 
we shall say we have a \DP-based 
duality between $\CA$ and $\CZ$.   
Dualities of this type occur widely in the literature, for many different varieties of $\CCD$-based algebras, not necessarily finitely generated, and have close affinities 
with the discrete dualities associated with Kripke-style relational semantics for various propositional logics. 
 We recall some preliminary  examples 
 to highlight some salient  points.  
The examples generalise Boolean algebras in different ways:
we consider a unary  operation modelling a non-classical negation 
and a binary operation modelling intuitionistic 
implication.

Let 
$\cat{DM}$ be the variety of De Morgan algebras (see for example \cite{BaDw,CF77}). 
The dual category $\cat{Z}_{\cat{DM}}$, of De Morgan spaces, consists of
the
Priestley spaces~$\Z$ carrying 
 a  continuous order-reversing involution, $g\colon \Z \to \Z$.  The lattice $\fnt K(\Z)$ 
obtained from a  De Morgan space $\Z$ carries a De Morgan negation given by
$(\neg a)(z) = 1 $ if and only if $a(g(z)) = 0$.   
The De Morgan space morphisms are the continuous order-preserving maps commuting with~$g$.  
The variety $\cat{K}$ of Kleene algebras is a subvariety of  $\cat{DM}$ and the associated dual category $\cat{Z}_{\cat{K}}$, of Kleene spaces, consists of those 
De Morgan spaces  $(Z, \leq,g,\Tp)$ such that 
for each $z \in Z$ we have $z \leq g(z) $ and/or $g(z) \leq z$ \cite{CF79}.   The varieties~$\cat{DM} $ and $\cat{K}$, besides being 
 of interest to logicians, 
have been influential in the development of the theory of coproducts
 originating 
 in \cite{CF77, CF79}, 
and of the theory we develop below.

For a contrasting example we consider   Heyting algebras.  This time 
the dual category consists of the so-called 
Esakia spaces, namely Priestley spaces~$\Z$ with the~property that, for each open set~$U$, the down-set  ${\downarrow}U$ generated by~$U$ is open.  With~$\fnt K(\Z)$ identified with the clopen up-sets of $\Z$, the Heyting implication is given by 
$a \to b = Z \setminus  {\downarrow}(a \setminus b)$.
A morphism of Esakia spaces is a continuous order-preserving map commuting with the 
${\downarrow}$ operator.  
(See 
for example 
\cite{DP96} for details.)

The very different forms taken by the dualities for De Morgan and Kleene algebras  and for Heyting algebras  signal  that dualities built on top of Priestley duality may be very diverse, and suggest that 
discovering them is more of an art than a science.   (Indeed, we are dealing here with a topic closely akin to 
correspondence theory, as the term is used in modal logic, 
and so should not expect a simple, all-embracing, method to be available for setting up dualities for $\CCD$-based algebras in general.) 
Already in the examples  above we can see 
 a 
 dichotomy emerging.  In the case of $\cat{DM}$ and $\cat{K}$,  the additional structure needed to capture negation is given by extra structure on the underlying Priestley spaces,  in the guise of the operation~$g$, whereas for Heyting algebras, the implication is uniquely determined by the order relation which is already present.  Thus in some cases coproducts in a class $\CA$ of $\CCD$-based algebras will be 
determined by 
their underlying lattices; in other cases additional work will be needed to describe coproducts completely.  

Our second observation is even more important.
 Given a dual equivalence 
between a $\CCD$-based category of algebras  $\CA$ and a 
$\CP$-based category $\CZ$, we have no guarantee that 
the cartesian product of $\CZ$-objects will be 
a {$\CZ$-object} (a classic example is provided by $\CA=\cat{K}$); 
 and even when this problem does not arise 
the projection maps might  not be $\CZ$-morphisms. 
  Therefore we contend that 
\DP-based 
dualities may be useful in analysing coproducts, but it 
cannot always
be expected that the product in the dual category will be as simple as a cartesian product.

We now turn our attention to natural duality theory and to multisorted piggyback dualities in particular.  
We  outline, in black box fashion, the framework we need in order to state the Multisorted Piggyback  Duality Theorem in the situation in which we employ it.  The theory is presented in full generality in 
\cite{DP87,CD98}.  
 We remark in passing  that Kleene algebras supplied the motivation and the trail-blazer example for 
the development of dualities employing multisorted dual structures
\cite{DP87}.

 We fix  until further notice a quasivariety  $\CA = \ISP(\CM)$ of $\CCD$-based algebras,
 where~$\CM$ is a finite set of finite algebras. 
The dual category $\CX$ is built using 
an \defn{alter ego}~$\CMT$ 
for~$\CM$; this  will be a multisorted topological structure.   
  We let the underlying set of~$\CMT$ 
be the (disjointified) union  $\dotbigcup {\CM}$ of the sets~$M$, for $\M\in \CM$, and its topology $\Tp$ be  the 
disjoint union topology obtained by equipping each set~$M$ 
with the discrete topology.  
 In addition, 
$\dotbigcup \CM$ is equipped with a set~$R$ of 
binary relations and a set~$G$ of unary maps between the members of~$\CM$.  
 Here we impose the restriction that  the relations in $R$  
are taken to be  \defn{algebraic}, meaning that 
each $r \in R$ is a subalgebra of $\M_1 \times \M_2$ 
 for some $\M_1, \M_2 \in \CM$ and that each member of~$G$ is a                                   
homomorphism from $\M_1$ to $\M_2$, for some $\M_1,\M_2$.

We must now describe how~$\CMT$, as above, is used to generate a category~$\CX$  
of  $\CM$-sorted topological structures.  
Every element~$\X$ of~$\CX$ will be a union 
$\dotbigcup_{\M \in \CM}\, X_{\M}$, 
where each sort $X_{\M}$ carries a topology and the union is equipped with the disjoint union topology.  
Moreover, $\X$ will be  equipped with a set of relations~$R^{\X}$ and operations~$G^\X$ indexed  by the members of $R$ and $G$. Obviously~$\CMT$  is of this type.
A \defn{morphism} from~$\X$ to~$\Y$, where ${\X, \Y \in\CX}$,  is a map which preserves the sorts (that is, it maps~$X_{\M}$ into~$Y_{\M}$ for each~$\M \in \CM$) and 
is  continuous and structure-preserving. 
 For a non-empty set~$S$, the $S\text{-fold}$ power of $\X\in \CX$
has underlying set $\dotbigcup_{\M \in \CM}\,  (X_{\M})^S$, with each $(X_{\M})^S$ equipped with the product topology;
 the relational structure is lifted pointwise in the obvious manner.  
    The candidate dual category $\CX:= \IScP(\CMT)$ for~$\CA$ 
then has as objects the $\CM$-sorted topological structures in~$\CX$ 
which are isomorphic copies of closed substructures 
of powers of~$\CMT$ and  the morphisms are
as described above; the empty structure is included. 

Finally in this preamble we specify the natural hom-functors, $\fnt D \colon \CA \to \CX$ and 
$\fnt E \colon \CX\to \CA$ based on $\CM$ and $\CMT$. 
  We define
\[
\fnt{D}(\A ) = 
\dotbigcupdisp 
\{\, \CA(\A,\M) 
 \mid \M \in \CM\,\}; 
\]
the sorts of $\X:=\fnt{D}(\A)$ are the hom-sets $X_{\M}:= \CA(\A,\M)$,
for $\M \in \CM$, and $\CA(\A,\M)$ is a closed subspace of the power~$M^A$, where $M$ is equipped with the discrete topology.  Each $r \in R$ is lifted pointwise to a relation 
$r^\X$, and $R^\X := \{\, r^\X \mid r \in R\,\}$, and 
$G^\X$ is defined likewise.
 Then $\fnt{D}(\A )$ belongs to~$\CX$.
Now consider any $\X \in \CX$.  A map from~$\X$ to~$\CMT$ 
which preserves the sorts can be seen as defining an element of the product 
\[
\textstyle\prod \{ \, M^{X_{\M}}  \mid \M \in \CM \, \}
\]
and the $\M$th  factor of this product can be structured pointwise from~$\M$ to give an element of~$\CA$.
Hence we can define   
\[
\fnt{E}(\X) = \CX  (\X,\CMT),
\] 
with~$\fnt{E}$ regarded as mapping~$\CX$ into~$\CA$.
We define~$\fnt D$ and~$\fnt E$ to act on morphisms by composition in the obvious way.
It is  then straightforward to verify that~$\fnt D$ and~$\fnt E$ are functors setting  up a dual adjunction between~$\CA$ and~$\CX$.

For each $\A \in \CA$ we have a \defn{multisorted evaluation map} 
$e_{\A} \colon \A \to \fnt{ED}(\A)$ given by
\[
e_{\A}(a)(\M)(x) = x(a)  \qquad \text{for } a \in A, \ x \in \CA(\A,\M).
\]
The map~$e_{\A}$ can then be shown to be  an embedding.  
We say that the alter ego~$\CMT$ of~$\CM$ \defn{yields a {\upshape(}multisorted\,{\upshape)} duality on}~$\CA$ if, for each 
  $\A \in \CA$, 
 the map 
~$e_{\A}$ is an isomorphism. 
In these circumstances, any algebra~$\A$ in~$\CA$, 
and in particular the coproduct of a family of algebras in~$\CA$,
 is determined up to isomorphism by its dual $\fnt{D}(\A)$, 
and can be recaptured from its dual as a family of continuous multisorted structure-preserving maps. 
An important special case concerns free algebras:
the dual $\fnt{D}({\Free}_{\CA}(S)) $ of the free algebra in~$\CA$ on a set~$S$ of free generators is the structure~$\CMT^S$ \cite[Lemma~2.2.1]{CD98}.  
(For the purposes of this paper we need only that $\fnt D$ and $\fnt E$
set up a duality between $\CA$ and
the subcategory 
$\fnt D(\CA)$ of $\CX =\IScP(\CMT)$, and not the
stronger requirement  that they set up a dual equivalence between $\CA$ and  $\CX $. For this reason we do not need to include partial operations in our alter ego $\CMT$.)

 The preceding  introduction does not  tell us how to choose the structure of 
the alter ego so as to make each evaluation map $e_{\A}$ 
an isomorphism.  
 The 
piggybacking technique  supplies an answer.  (The term  `piggyback'
is used because the proof 
 relies on the 
known 
dual equivalence between 
$\CCD $ and $\CP$
to set up a suitable alter ego.)  
The piggybacking theorem, essentially as we give it below,
originated  in \cite[Section~2]{DP87} and is
reproduced as \cite[Theorem~7.2.1]{CD98}.
Before stating the theorem we introduce 
notation to describe the relations we shall use. 
 For $\M_i \in \CM$   
and each 
$\w_i \in \CCD(\fnt U(\M_i),\two)$ ($i = 1,2$), we have 
a bounded sublattice 
  \[
(\w_1,\w_2)^{-1}(\leqslant) :=
\{\,  (a,b) \in \M_1 \times \M_2\mid \w_1 (a) \leq \w_2 (b)\,\} . 
\]
of~$\M_1 \times \M_2$.  
It
will contain a (possibly empty) set~$R_{\w_1,\w_2}$ of algebraic relations 
 which are maximal with respect to the property of being contained in  
$(\omega_1,\omega_2)^{-1}(\leqslant)$.   
 
\begin{theorem} \label{genpig} {\rm(Multisorted Piggyback Duality Theorem, for 
$\CCD$-based algebras)} 
 Let  
$\CA  = \ISP (\CM)$, where~$\CM$ is a finite set of finite $\CCD$-based 
algebras of common type. 
For each $\M \in \CM$  let~$\Omega _{\M}$ be a {\rm(}possibly
empty{\rm)} subset of $  
\CCD(\fnt U(\M), \two)$.  
Let 
\[
\CMT   = \langle \,
\dotbigcupdisp 
_{\M \in \CM}\,  M ; 
R, G, \Tp\, \rangle 
\] 
 be the topological  structure
in which 
\begin{newlist}
\item[{\rm(i)}] $\Tp$ is the discrete topology;
\item[{\rm(ii)}] 
\begin{newlist}
\item[{\rm(a)}] $R$ is  the 
union, for $\w_i  \in \Omega _{\M_i}$ 
{\upshape (}$i \in \{1,2\}${\upshape)} of the sets $R_{\w_1,\w_2}$, as  defined above,
and  
\item[{\rm(b)}] $G
= 
\bigcup \{\, \CA (\M_1,\M_2)\mid \M_1, \M_2 \in \CM\,\}$.
\end{newlist}
\end{newlist}
Let $\Omega=\bigcup_{\M\in\CM}\Omega_{\M}$ and assume that 
the sets~$\Omega_{\M}$ 
 have been chosen so that 
the following separation condition is satisfied:
\begin{newlist}
\item[$\text{\rm (Sep)}_{\CM,\Omega}$:] for all $\M\in \CM$, given 
$a,b\in \M$  with $a \ne b$, there exist 
$\M' \in \CM$, \newline \hspace*{1cm} $u\in \CA (\M, \M') $ 
 and 
$\w   \in \Omega _{\M'}$  such that 
$\w  (u(a))  \ne \w (u(b))$. 
\end{newlist}
Then~$\CMT$ yields a multisorted duality on~$\CA$.
\end{theorem}                                                                      

Given~$\CM$,   the separation condition 
$\text{\rm (Sep)}_{\CM,\Omega}$ can always be satisfied by taking
$\Omega_{\M} 
=\CCD(\fnt U(\M), \two)$ for each $\M \in \CM$.

The elements of $\Omega$ are called  \defn{carrier maps}.
In the separation condition we shall  omit set brackets when $\CM$  contains a single  element, and likewise for~$\Omega$.
The term \defn{simple piggybacking} is used when, given $\CA$, there exists a 
finite algebra $\M$ with $\CA =\ISP(\M)$ and a single carrier map 
$\w \in \CCD(\fnt U(\M),\two)$ such that 
 $\text{(Sep)}_{\M,\w}$ holds.  
   Examples of simple piggyback dualities are given in~\cite{DWpig} and of multisorted piggyback dualities in \cite{DP87, CD98},
as well as in Section~\ref{Sec_Ex}.  
 Simple piggybacking is not always possible---the variety
of Kleene algebras provides an archetypal example; see 
 \cite{DP87}.   When $|\CM| = 1$, we shall say that the piggyback
duality supplied by Theorem~\ref{genpig} is \defn{single-sorted}. 

Assume that $\CM$ and $\Omega= \bigcup \Omega_\M$ are such that
$\CA = \ISP(\CM)$ and $\text{(Sep)}_{\CM,\Omega}$ holds. 
There is an intimate relationship between the elements $\M \in \CM$
and the associated  carrier maps as regards the representation of $\CA$ as a quasivariety.  
Whenever $\M\in\CM$ is such that $\Omega_{\M}=\emptyset$, condition ${\rm (Sep)}_{\CM,\Omega}$ implies, on the one hand, that ${\rm (Sep)}_{\CM\setminus\{\M\},\Omega}$ holds. On the other hand,  for each  $a \ne b$ in~$\M$,
we can find an algebra  $\M' \in \CM\setminus \{ \M'\}$, and maps $u \in \CA(\M,\M')$  and $\omega \in \Omega_{\M'}$ such that 
$\w(u(a)) \ne \w(u(b))$ and hence necessarily $u(a) \ne u(b)$.  Then $\M \in \ISP(\CM \setminus \{ \M'\})$, and this together with $\CA = \ISP(\CM)$ imply that $\CA = \ISP(\CM \setminus \{ \M'\})$. In particular we see that if $\text{(Sep)}_{\CM,\Omega}$
holds, then only those members $\M$ of $\CM$ for which $\Omega_{\M}$  is non-empty are needed when $\CM$ is used to represent $\CA$ as a finitely generated quasivariety, and to develop a piggyback duality as in Theorem~\ref{genpig}.
We shall make  crucial use of this elementary observation in the proof of 
Theorem~\ref{Theo:OntoCop} below.
(In Theorem~\ref{genpig} we took,  as the default,  $G$ to contain all
homomorphisms between members of $\CMT$.   Depending on the choice of $\Omega$, 
a smaller set $G$ may suffice.)

We are now almost ready to achieve our stated goal  of relating  the Priestley space $\fnt{HU}(\A)$ to the dual space $\fnt D(\A)$ of $\A\in \CA$ under the duality 
set up as in Theorem~\ref{genpig}. In the 
proof of Theorem~\ref{Theo:RevEng}  and several later proofs we  make 
use of the Joint Surjectivity  Lemma (see \cite[Lemma 7.2.2]{CD98}) which is a key ingredient in the proof of the Multisorted Piggybacking Theorem.  
For convenience we recall this lemma below.  
 Observe that condition~(3)
says that every element of the natural dual $\fnt{HU}(\A)$ can be realised
as the image of  an element in $\fnt D(\A)$ under   a map of the form $\w \circ -$.  Thus the lemma can be seen as a first step towards
relating 
$\fnt{HU}(\A)$  to the  multisorted dual $\fnt D(\A)$.

\begin{lemma} \label{jointsurj}
{\rm (Joint Surjectivity)}
Let $\CA=\ISP(\CM)$ and let $\Omega = \bigcup \,\Omega_\M$, where  let~$\Omega _{\M}$ is a {\rm(}possibly
empty{\rm)} subset of $\CCD(\fnt U(\M), \two)$,
 for each $\M\in\CM$.
Then the following 
conditions are equivalent:
\begin{newlist}
\item[{\rm (1)}] condition $\text{\rm (Sep)}_{\CM,\Omega}$, as given in 
Theorem~{\upshape\ref{genpig}}, is satisfied;
\item[{\rm (2)}] for every $\A\in\CA$ and  every $a, b \in \A$ with $a \ne b$ there exist $\M \in \CM$, $\w \in \Omega_{\M}$ and $x \in \CA(\A,\M)$ such that 
$\w(x(a)) \ne \w(x(b))$;
\item[{\rm (3)}]   for every $\A\in\CA$ and each $y \in \fnt H\fnt U(\A)$ there exist
$\M \in \CM$, $\w \in \Omega_{\M}$ and $x \in  \CA(\A,\M)$ such that
$y = \w \circ x$.
\end{newlist} 
\end{lemma}

For the remainder of this section we  assume that $\Omega$ and $R$ have  been chosen so that the conditions of  Theorem~\ref{genpig} are satisfied.  
We  establish some  additional notation.
For a fixed algebra $\A \in \ISP(\CM)$  and $\str{X}=\fnt D(\A)$, we 
let
$$
  \textstyle  Y=\bigcup\{\, X_{\M}\times\Omega_{\M}\mid \M\in\CM\,\}
$$ 
and equip it  
with the topology $\Tp^{Y}$ having as 
a base 
of open sets
$$
    \{\,U\times\{\w\}\mid  U\mbox{ open in } X_{\M}\mbox{ and }\w\in\Omega_{\M}\,\}
$$
and  with 
 the binary relation  $\preceq\, \,\subseteq Y^2$ defined by 
$$
    (x,\w_1)\preceq(y,\w_2)\mbox{ if }(x,y)\in r^{\X}\mbox{ for some }r\in R_{\w_1,\w_2}.  %^{\X}.  
$$

\begin{theorem}\label{Theo:RevEng}
For each $\A\in \CA$, the structure $(Y,\preceq,\Tp_Y)$ defined above 
has the following properties.
  \begin{itemize}
    \item[{\rm (i)}] The 
space $(\, Y,\Tp_Y)$ is compact and Hausdorff.
    \item[{\rm (ii)}] The binary relation $\preceq\, \subseteq Y^2$ is a pre-order on $Y$.
    \item[{\rm (iii)}]  Let $\approx\,=\,\preceq\cap\succcurlyeq$ 
				    be the equivalence relation on $Y$ determined by $\preceq$. Then 
$$
  (\, Y/_{\approx},{\preceq}/_{\approx},\Tp^Y\!/_{\approx})
$$
is a Priestley space isomorphic
 to $\fntH \fnt U(\A) $, where $\Tp^Y\!/_{\approx}$ denotes the quotient topology.
  \end{itemize}
\end{theorem}

\begin{proof}  
We recall that for each  for $\M \in \CM$ the $\M$th sort of the multisorted structure
$\fnt D(\A)$ is $X_{\M} := \CA(\A, \M)$.
The topology of $Y$ coincides with the topology of the finite disjoint union of the topological spaces $X_{\M}\times\Omega_{\M}$,
 considering $X_{\M}$ as carrying the induced topology from $(X,\Tp^{\X})$ and $\Omega_{\M}$ the discrete topology. 
Then (i) follows from the fact that $\Omega$ is finite.

Now consider (ii). 
For each $\M_1,\M_2\in\CM$, each $x\in X_{\M_1}$ and $y\in X_{\M_2}$, and each $\w_i\in\Omega_{\M_i}$ (with $i\in\{1,2\}$),
\begin{align*} 
(x,\w_1)\preceq(y,\w_2) & \Longleftrightarrow (x,y)\in r^{\X} \text{ for some } r\in R_{\w_1,\w_2}^{\X}\\
& \Longleftrightarrow  \{(x(a),y(a))\mid a\in {\A}\}\subseteq r \text{ for some } r\in R_{\w_1,\w_2}\\
& \Longleftrightarrow  \{(x(a),y(a))\mid a\in {\A}\}\subseteq (\w_1,\w_2)^{-1}(\leq)\\
    & \Longleftrightarrow   (\w_1\circ x)(a)\leq (\w_2\circ y)(a) \text{ for each } a\in {\A}\\ 
  & \Longleftrightarrow \w_1\circ x\leq \w_2\circ y \text{ in } \fntH\fnt U(\A).
\end{align*}
It is straightforward to check that $\preceq_{Y_{\A}}$ is reflexive and transitive, and therefore a pre-order, which proves (ii).

For (iii) we need to set up the required isomorphism between the quotient structure obtained in~(ii) and $\fntH \fnt U(\A)$. 
Let $\Phi\colon Y\to \fntH \fnt U(\A) $ 
be
defined by $\Phi(x,\w)=\w\circ x$. From the characterisation of $\preceq$ above 
it follows that 
\[
(x,\w_1)\preceq_{Y_{\A}}  (y,\w_2)
\Longleftrightarrow 
\Phi(x,\w_1)\leq \Phi(y,\w_2),
\]
for each $(x,\w_1),(y,\w_2)\in Y$. 
Thus $\ker(\Phi)$ equals~$\approx$.   The unique map 
\[
\Phi'\colon Y/_{\approx}\to \fntH\fnt U(\A)\]
such that $\Phi=\Phi'\circ\rho$ is
 order-preserving and 
order-reversing; 
 here $\rho\colon Y\to {Y}/_{\approx}$ denotes the quotient map $\rho(x,\w)=[(x,\w)]_{\approx}$; see Fig.~\ref{Fig:HAfromDA}.   It follows that $\Phi'$ is injective.
 \begin{figure} [ht]
\begin{center}
\begin{tikzpicture} 
[auto, 
 text depth=0.25ex,
] 
\matrix[row sep= .9cm, column sep= .9cm]
{
\node (YA) {$(Y,\preceq,\Tp^Y)$}; & &\node (HA) {$\fntH\fnt U(\A) $};\\
 \node (YAQ) {$(\, {Y}/_{\approx},{\preceq}/_{\approx},\Tp^{Y}\!/_{\approx})$};&\\ 
};
\draw [->>] (YA) to node [swap] {$\rho$} (YAQ);
\draw [->] (YA) to node  {$\Phi$} (HA);
\draw [->] (YAQ) to node [swap] {$\Phi'$} (HA);

\end{tikzpicture}
\end{center}\caption{The proof of Theorem~\ref{Theo:RevEng}}\label{Fig:HAfromDA}
\end{figure} 

To prove that $\Phi'$ is also surjective, we appeal to Lemma~\ref{jointsurj}.  Condition (2) in the lemma holds because 
 $\CA = \ISP(\CM)$ and  condition $\text{(Sep)}_{\CM,\Omega}$ is assumed to hold. 
Therefore  $\Phi'$ and $\Phi$ are surjective maps.

Given 
$a\in \A$, let $\pi_a\colon y \mapsto y(a)$ be  the evaluation map from $\fntH\fnt U(\A)$ into~$\str 2$. 
Then
\begin{align*}
    (\pi_a\circ\Phi)^{-1}(\{1\})& =\{\,(x,\w)\in Y\mid (\w\circ x)(a)=1\,\}\\
   &=\{\,(x,\w)\in Y\mid x\in (e_{\A}(M)(a))^{-1}(\w^{-1}(\{1\}))\mbox{ and } \w\in \Omega_{\M}\,\}\\
   &\textstyle=\bigcup\{\,e_{\A}(M)(a)^{-1}(\w^{-1}(\{1\}))\times \{\w\}\mid \M\in \CM \mbox{ and }\w\in \Omega_{\M}\,\}.
\end{align*}
Since $e_{\A}(a)$ is continuous for each $a\in\A$ it follows that $e_{\A}(a)^{-1}(\w^{-1}(\{1\}))$ is open in $X_{\M}$, 
and therefore $(\pi_a\circ\Phi)^{-1}(\{1\})$ is open in $\Tp^Y$. 
Similarly, we have that $(\pi_a\circ\Phi)^{-1}(\{0\})$ is open, proving that $\pi_a\circ\Phi$ is continuous for each $a\in A$.
Since $\fntH\fnt U(\A) $ inherits  its 
 topology from $\stwiddle{\two}^{A}$, it follows that $\Phi$ is continuous.
By the definition of the quotient topology,
$\Phi'$ is continuous.

Combining  (i) with the fact that the quotient map $\rho$ is continuous, 
we obtain that $(\, {Y}/_{\approx},\Tp^Y\!/_{\approx})$ is  compact. 
Since $\Phi'$ is 
an injective  continuous map from  $(\, {Y}/_{\approx},\Tp^Y\!/_{\approx})$ 
onto the
Hausdorff  space $\fntH \fnt U(\A)$ and hence  that $\Phi'$ is a homeomorphism. 
Therefore 
$\Phi'$ is an isomorphism of Priestley spaces.
\end{proof}

A special case of Theorem~\ref{Theo:RevEng} deserves to be recorded  as a corollary.  It describes the situation which pertains, in particular, whenever coproducts are preserved,
as Theorems~\ref{MainTheo} 
and~\ref{Theo:CoproToNatDual} will show. 

\begin{corollary} \label{Cor-to-RevEng}
Assume $\CA = \ISP(\M)$ where $\M$ is a finite $\CCD$-based algebra 
and that $\text{\rm (Sep)}_{\M,\w}$ holds for some $\w \in \fnt H\fnt U(\M)$, so that $\CA$ has a simple piggyback duality. If  $|R_{\w,\w}|=1$, then for each $\A \in \CA$, the Priestley 
space $\fnt H\fnt U(\A)$ is isomorphic to the ordered space
obtained from the natural dual space $\X = \fnt D(\A)$ by equipping 
its underlying set with the partial order relation $ R_{\w,\w}^\X$ and its existing topology.  
\end{corollary}

The corollary can be regarded as asserting that, under the stipulated
conditions, the quasivariety $\CA$ has a natural duality which is, in a 
certain sense, also a \DP-based
 duality.  In addition, if we
have already derived a  \DP-based duality for $\CA$, then we can regard the structure on the Priestley dual spaces which is used to capture any additional operations as living also on the natural dual spaces.
Thus there is a very close relationship between a \DP-based 
duality for~$\CA$ and the piggyback duality mentioned in the corollary.

One additional remark deserves to be made about the special situation 
covered by Corollary~\ref{Cor-to-RevEng}.  In general, some of the sets $R_{\w_1,\w_2}$, for $\w_1,\w_2 \in \Omega$, arising in 
Theorem~\ref{Theo:RevEng} may be empty (and we shall see instances
of this in Section~\ref{Sec_Ex}).  But when Corollary~\ref{Cor-to-RevEng}
applies, then $R_{\w,\w}$ is necessarily non-empty.

We note that in case we have $\CA = \ISP(\M)$ but of necessity 
more than one carrier map, there is a choice as to how to formulate 
a natural duality for $\CA$.  If we base the natural duality on the  single 
algebra~$\M$ then the dual spaces will be structured Boolean spaces,
with the structure lifted from an alter ego with underlying set $M$.  The 
construction in Theorem~\ref{Theo:RevEng} then tells us that, for $\A \in \CA$ we get 
the Priestley dual space
 of the underlying lattice $\fnt U(\A)$ by first taking multiple copies of 
$\fnt D(\A)$.  If, instead, we opt to set up the natural duality with 
multiple copies of $\M$ and a single carrier map associated with each,
then we essentially build the first stage of the translation process into the 
natural duality.  Which of these two approaches is to be preferred is 
largely a matter of taste.  In Section~\ref{Sec_Ex} we adopt the former
approach, noting that the second has been used in the literature in certain cases.

\section{Coproducts}\label{Sec_Cop}

In this section we  embark on discussion of the central topic of this paper
and  
present our main results.
 Specifically, we shall describe the structure of coproducts in finitely generated $\CCD$-based classes of algebras 
in terms of 
the corresponding coproducts in $\CCD$.
 Background material on  
coproducts in a
categorical context can be found in \cite{McL1969}.

 Let 
$\CA$ be a quasivariety of algebras, not yet assumed
 to be finitely generated.  
Assume that $\CA$ is $\CCD$-based, and that $\fnt{U} \colon \CA \to \CCD$
is the associated forgetful functor.
Let  $\class{K}$ be a set of algebras in $\CA$.
Since $\CA$ is a quasivariety, the coproduct $\coprod_{\CA} \class{K}$ in $\CA$ 
 always exists.  
In what follows, we shall omit the subscript in $\coprod_{\CA}$ when the class where we are considering the coproduct is clearly determined by context.
Let $\{\,\epsilon_{\B}\colon \B\to\coprod \class{K}\mid \B\in \class{K}\,\}$
 be the universal co-cone that determines the coproduct $\coprod \class{K}$ up to isomorphism 
(Fig.~\ref{fig:coprod}(a)). 
Then $\fnt U$ \defn{preserves coproducts} if and only if,  for each set $\class{K}\subseteq\CA$,
\[\textstyle
\{\,\fnt U(\epsilon_{\A})\colon \fnt U(\A)\to\fnt U(\coprod(\class{K}))\mid \A\in \class{K}\,\}
\]
 is a universal co-cone in $\CCD$.
There always exists a unique map 
\[
  \textstyle\chisub{\class{K}}\colon \coprod\fnt U(\class{K})\to \fnt U(\coprod\class{K})
\]
such that $\chisub{\class{K}}\, \circ\, \varepsilon_{\B}=\fnt U(\epsilon_{\B})$ for each $\B\in\class{K}$, where the family of maps
$  \{\varepsilon_{\B}\colon \fnt U(\B)\to \coprod\fnt U(\class{K})\mid \B\in\class{K}\}$ is the universal co-cone in $\CCD$ for $\coprod\fnt U(\class{K})$ 
(Fig.~\ref{fig:coprod}(b)).
Proving that $\fnt U$ preserves coproducts is equivalent to proving that, for each set of algebras $\class{K}$, the homomorphism $\chisub{\class{K}}$ is an isomorphism.
This in turn  is equivalent to proving that~$\fntH(\chi_{\class{K}})$ is an isomorphism (in the category of Priestley spaces) between $\fntH\fnt U(\coprod \class{K})$ and $\fntH (\coprod \fnt U(\class{K}))$.

Since 
$\fntH\colon \CCD\to \CP$ sends coproducts into cartesian products,  we can identify $\fntH (\coprod \fnt U(\class{K}))$ with $\prod \fntH\fnt U(\class{K})$ for each set 
$\class{K}\subseteq \CA$. Under this identification,
\[
  \textstyle\iota_{\class{K}}=\fntH(\chisub{\class{K}})\colon \fntH\fnt U(\coprod \class{K})\to \prod\fntH\fnt U(\class{K})
\]
is the unique map such that $\pi_{\fntH\fnt U(\B)}\circ\iota_{\class{K}}=\fntH\fnt U(\epsilon_{\B})$ for each $\B\in\class{K}$, where 
\[
  \textstyle\pi_{\fntH \fnt U(\B)}=\fntH(\varepsilon_{\B})\colon \prod\fntH\fnt U(\class{K})\to \fntH\fnt U(\B)
\]
 denotes the projection map (Fig.~\ref{fig:coprod}(c)).
In order to determine when $\iota_{\class{K}}$ is an isomorphism of Priestley spaces we shall first describe this map using Theorem~\ref{Theo:RevEng} (see Lemma~\ref{Lem:FactMap}).

\begin{figure} [ht] 
\begin{center}
\begin{tikzpicture} 
[auto,
 text depth=0.25ex,
 move up/.style=   {transform canvas={yshift=2.5pt}},
 move down/.style= {transform canvas={yshift=-2pt}},
 move left/.style= {transform canvas={xshift=-2.5pt}},
 move right/.style={transform canvas={xshift=2.5pt}}] 
\matrix[row sep= .9cm, column sep= .6cm]
{
\node (Cop) {$\coprod\class{K}$}; 
& \node (CU) {$\coprod\fnt U(\class{K})$};  
&&\node (UC) {$\fnt U(\coprod\class{K})$};
& \node (PH) {$\prod\fntH\fnt U(\class{K})$};
& &  \node (HCop) {$\fntH\fnt U(\coprod\class{K})$}; \\
   \node (Alg) {$\B$}; 
& &  &\node (UAlg) {$\fnt U(\B)$};
& & 
&\node (HAlg) {$\fntH\fnt U(\B)$}; \\[-.4cm] 
{\hspace{-.2cm}${\rm (a)}\quad $}; &&{$ {\rm (b)}$};&&&&{\hspace{-.7cm}${\rm (c)}$};\\[-.3cm]
};
\draw [->] (Alg) to node  {$\epsilon_{\B}$} (Cop);
\draw [->] (UAlg) to node  {$\varepsilon_{\B}$} (CU);
\draw [->] (UAlg) to node [swap] {$\fnt U(\epsilon_{\B})$} (UC);
\draw [->] (CU) to node {$\chisub{\class{K}}$}(UC);
\draw [->] (HCop) to node  {$\fntH\fnt U(\epsilon_{\B})$} (HAlg);
\draw [->>] (PH) to node [swap] {$\pi_{\fntH\fnt U(\B)}$} (HAlg);
\draw [->] (HCop) to node [swap]{$\iota_{\class{K}}$}(PH);

\end{tikzpicture}
\end{center}
\caption{Universal mapping definitions of $\chisub{\class{K}}$ and $\iota_{\class{K}}$}\label{fig:coprod}
\end{figure} 
We   
fix until further notice a  $\CCD$-based quasivariety $\CA$ 
which is expressible as~$\CA = \ISP(\CM)$, where $\CM$ is a finite set of finite algebras.
Adopting  
the notation 
introduced 
 in 
Theorem~\ref{Theo:RevEng}, we have 
\[
\textstyle Y_{\coprod \class{K}}=\bigcup\{\,\fnt{D}(\coprod \class{K})_{\M}\times\Omega_{\M}\mid \M\in\CM\,\}.
\]
We now use the fact 
 that $\fnt{D}(\coprod \class{K})\cong \prod\fnt{D}(\class{K})= \prod \{\,\fnt{D}(\B)\mid \B\in \class{K}\,\}$ and that, under this isomorphism,  $\fnt D(\epsilon_{\B})\colon \prod \{\,\fnt{D}(\B)\mid \B\in \class{K}\,\}\to \fnt D(\B)$ is the projection map. 
Hence for each $\B\in \class{K}$ we have  
a map   
$\xi_{\B}\colon Y_{\coprod \class{K}}\to Y_{\B}$ defined by 
\[
\xi_{\B}(\vec{x},\w)= 
(\fnt D(\epsilon_{\B})(\vec{x}),\w)=(x_{\B},\w).
\] 
Since $\fnt D(\epsilon_{\B})$ preserves the relations in $R$,  it follows that,   if $(\vec{x},\w)\preceq_{\coprod\class{K}}(\vec{y},\w')$ then $\xi_{\B}(\vec{x},\w)\preceq_{\B}\xi_{\B}(\vec{y},\w')$. Moreover $\xi_{\B}$ is continuous for each $\B\in \class{K}$.
We 
can then
define a 
 map 
$\Psi_{\coprod \class{K}}\colon Y_{\coprod \class{K}}\to 
 \prod\fntH\fnt U(\class{K})$ by 
$$
\textstyle (\vec{x},\w)\in\prod\fnt{D}(\class{K})\times\Omega_{\M}\mapsto
 \vec{y}\in\prod \fntH\fnt U(\class{K}),
$$ 
where $y_{\B}=\w\circ x_{\B}=\Phi_{\B}(\xi_{\B}(\vec{x},\w))$ for each $\B\in \class{K}$ and for each $\M\in \CM$. 
It follows that 
$$
  (\vec{x}_1,\w_1)\preceq_{\coprod \class{K}} (\vec{x}_2,\w_2)
  \Longrightarrow \Psi_{\coprod \class{K}}(\vec{x}_1,\w_1)\leq \Psi_{\coprod \class{K}}(\vec{x}_2,\w_2).
$$
In fact, if $(\vec{x}_1,\w_1)\preceq_{\coprod \class{K}} (\vec{x}_2,\w_2)$,
then there exists 
$r \in R_{\w_1,\w_2}$ 
for which 
$(\vec{x}_1,\vec{x}_2)\in r$ on $\fnt D(\coprod \class{K})$.
Therefore  $((x_1)_{\B},(x_2)_{\B})\in r$ 
on $\fnt D(\B)$
for each $\B\in \class{K}$. 
Then, 
for each $\B\in\class{K}$, the maps
$\xi_{\B}$ and $\Psi_{\coprod\class{K}}$ are such that the diagram in Fig.~{\upshape \ref{Fig:DefPsi}}
commutes.
\begin{figure} [ht]
\begin{center}
\begin{tikzpicture} 
[auto,
 text depth=0.25ex,
 move up/.style=   {transform canvas={yshift=2.5pt}},
 move down/.style= {transform canvas={yshift=-2pt}},
 move left/.style= {transform canvas={xshift=-2.5pt}},
 move right/.style={transform canvas={xshift=2.5pt}}] 
\matrix[row sep= .9cm, column sep= .9cm]
{ 
\node (YCop) {$Y_{\coprod\class{K}}$}; & \node (PH) {$\prod\fntH\fnt U(\class{K})$}; \\
    \node (YAlg) {$Y_{\B}$};& \node (HAlg) {$\fntH\fnt U(\B)$};\\ 
};
\draw [->] (YCop) to node [swap]  {$\xi_{\B}$}(YAlg);
\draw [->] (PH) to node  {$\pi_{\fntH\fnt U(\B)}$} (HAlg);
\draw [->] (YCop) to node {$\Psi_{\coprod\class{K}}$} (PH);
\draw [->>] (YAlg) to node [swap] {$\Phi_{\B}$}(HAlg);

\end{tikzpicture}
\end{center}\caption{The definition of $\Psi_{\coprod\class{K}}$ }\label{Fig:DefPsi}
\end{figure} 

By applying Theorem~\ref{Theo:RevEng} to $\coprod \class{K}$ 
we can obtain 
the following lemma.

\begin{lemma}\label{Lem:FactMap}
For each set of algebras $\class{K}\subseteq \CA$, the map 
$$
  \textstyle\iota_{\class{K}}\colon \fntH\fnt U(\coprod \class{K})\to \prod\fntH\fnt U(\class{K})
$$
satisfies  
$ 
\iota_{\class{K}}\circ \Phi_{\coprod \class{K}}=\Psi_{\coprod \class{K}}
$. 
That is,  the diagram in Fig.~{\upshape \ref{Fig:Factor}}
commutes.  
\begin{figure} [ht]
\begin{center}
\begin{tikzpicture} 
[auto,
 text depth=0.25ex,
 move up/.style=   {transform canvas={yshift=2.5pt}},
 move down/.style= {transform canvas={yshift=-2pt}},
 move left/.style= {transform canvas={xshift=-2.5pt}},
 move right/.style={transform canvas={xshift=2.5pt}}] 
\matrix[row sep= .9cm, column sep= .9cm]
{  & & \node (PH) {$\prod\fntH\fnt U(\class{K})$};\\
\node (YCop) {$Y_{\coprod\class{K}}$}; & \node (HCop) {$\fntH\fnt U(\coprod\class{K})$}; \\
    \node (YAlg) {$Y_{\B}$};& \node (HAlg) {$\fntH\fnt U(\B)$};\\ 
};
\draw [->] (YCop) to node [swap]  {$\xi_{\B}$}(YAlg);
\draw [->] (HCop) to node  {$\fntH\fnt U(\epsilon_{\B})$} (HAlg);
\draw [->>] (YCop) to node {$\Phi_{\coprod\class{K}}$} (HCop);
\draw [->>] (YAlg) to node [swap] {$\Phi_{\B}$}(HAlg);
\draw [->] (HCop) to node {$\iota_{\class{K}}$}(PH);
\draw [->](YCop) .. controls +(.5,1) and +(-2,.2) .. node {$\Psi_{\coprod\class{K}}$} (PH);
\draw [->>] (PH) .. controls +(-.6,-2.5) and +(2,0) .. node {$\pi_{\fntH\fnt U(\B)}$} (HAlg);

\end{tikzpicture}
\end{center}\caption{Lemma~\ref{Lem:FactMap}}\label{Fig:Factor}
\end{figure} 
\end{lemma}

\begin{proof}
Let $(\vec{x},\w)\in Y_{\coprod \class{K}}$, then $(\vec{x},\w)\in(\prod \fnt{D}(\class{K})_{\M})\times\Omega_{\M}$ for some $\M\in\CM$. 
Then, for each $\B\in\class{K}$, 
\begin{align*}
(\pi_{\fntH\fnt U(\B)}\circ\iota_{\class{K}}\circ \Phi_{\coprod \class{K}})(\vec{x},\w)&=(\pi_{\fntH\fnt U(\B)}\circ\iota_{\class{K}})(\w\circ \vec{x}) =\fntH\fnt U(\epsilon_{\B})(\w\circ \vec{x})\\
&= (\pi_{\fntH\fnt U(\B)} \circ\Psi_{\coprod \class{K}})(\vec{x},\w). \qedhere
\end{align*} 
\end{proof}

By  Lemma~\ref{jointsurj}, 
 $  \{\,\w\circ x\mid x\in \fnt{D}(\B)\mbox{ and }\w\in\Omega\,\}=\fntH\fnt U(\B).
$
Hence  the assignment 
\[
f 
\mapsto\Lambda_{\B}(f) :=\{\,\w\in\Omega\mid f=\w\circ x\mbox{ for some } x\in \fnt D(\B)\}
\] 
gives  a well-defined map $\Lambda_{\B}$ from 
$\fntH\fnt U(\B)$ into the family of non-empty subsets
of $\Omega$.
With this notation the following result follows directly from Theorem~\ref{Theo:RevEng}.

\begin{corollary}\label{Cor:Range}
For every set of algebras $\class{K}\subseteq \CA$,
$$\textstyle
\iota_{\class{K}}\big(\fntH\fnt U(\coprod \class{K})\big)=\{\,\vec{y}\in \prod\fntH\fnt U(\class{K})\mid \bigcap\{\,\Lambda_{\B}(y_{\B})\mid \B\in \class{K}\,\}\neq\emptyset\,\}.
$$  
\end{corollary}

We now have in place the  relationships which we can establish in general about the maps
  $\textstyle\chisub{\class{K}}\colon \coprod\fnt U(\class{K})\to \fnt U(\coprod\class{K})$, for a subset $\class{K}$ of our given quasivariety~$
\CA$.  
We 
progress to an investigation of conditions under which these maps are necessarily isomorphisms. We split the problem into two natural parts.   
We say that the functor $\fnt U$ satisfies
\begin{newlist}
\item[(S)] if for each set $\class{K}\subseteq\CA$, the map $\chisub{\class{K}}$ is surjective;
\item[(E)] if for each set $\class{K}\subseteq\CA$, the map $\chisub{\class{K}}$ is injective.
\end{newlist}

In preparation for the following results we  establish some conventions.
 Given a class of algebras $\cat{B}$, for each $\str{B}\in\cat{B}$ and $C\subseteq \B$, we write  $\langle C\rangle_{\cat{B}}$ to denote the subalgebra of $\B$ generated by $C$. 
We shall repeatedly encounter sets of subalgebras of 
some specified algebra~$\A$.  We shall always order such sets  by inclusion.
Under this assumption, the assertion that 
$Y\subseteq\ope{S}(\A)$ has a top element is equivalent to saying that $\bigcup Y$ is the universe of a subalgebra 
of~$\A$.

 Our next step is to determine when $\chisub{\class{K}}$ maps $\coprod\fnt{U}( \class{K})$ 
onto $\fnt U(\coprod\class{K})$  or, equivalently,  when $\iota_{\class{K}}$ is an order embedding.

\begin{theorem}\label{Theo:EmbeddingCoprod}  Assume that 
$\CA$ is a finitely generated $\CCD$-based quasivariety.
Assume that $\CM$ is such that $\CA = \ISP(\CM)$ and that 
$\Omega \subseteq \dotbigcup_{\M \in \CM}  \fnt{HU}(\M)$ satisfies ${\rm (Sep)}_{\CM,\Omega}$.
Then
  the following statements are equivalent:
  \begin{newlist}
    \item[{\rm (1)}] $\fnt U$ satisfies {\rm (S)},  that is, for every set  $\class{K} $ of algebras in $\CA$, the image of $\chisub{\class{ K}}$ is 
$\fnt U(\coprod\class{K})$;
    \item[{\rm (2)}] for every finite set   $\class{K}$ of disjoint copies of 
$\Free_{\CA}(1)$, 
the image of $\chisub{\class{ K}}$ is 
$\fnt U(\coprod\class{K})$;
\item[{\rm (3)}] 
for every $n\geq 1$ and every  $n$-ary term $t$ in the language $\mathcal L$  of $\CA$ there exist unary terms 
$t_1,\ldots,t_n$ 
in  
 $\mathcal L$ and  an $n$-ary term  $s$ in the language of $\CCD$ such that 
$$
t^{\A}(a_1,\ldots,a_n)=s^{\A}(t_1^{\A}(a_1),\ldots,t_n^{\A}(a_n)),
$$
 for every $\A\in\CA$ and every $a_1,\ldots,a_n\in\A$;
\item[{\rm (4)}]
for every $\A\in\CA$ and  every    bounded sublattice  $\str{L}$ of $\fnt U(\A)$, then either $L$ contains no $\CA$-subalgebra of 
$\A$  or 
$\{\,\B\in\ope{S}(\A)\mid B\subseteq L\,\}$ 
has a top element;
\item[{\rm (5)}]  for every $\w_1, \w_2\in \Omega$
it is the case that  
  $|R_{\w_1,\w_2}|\leq 1$;
\item[{\rm (6)}] for every set of algebras $\class{K}\subseteq \CA$, the map $\iota_{\class{ K}}$ is an order embedding.
\end{newlist}
\end{theorem}
\begin{proof}
The equivalence of 
 (1) and  (6) is a standard fact about Priestley duality.
We obviously have (1) $\Rightarrow $  (2).

We  first  prove that 
(2) $\Rightarrow$ (3).  
Let $t$ be a $n$-ary term in the language~$\mathcal L$ and let~$\class{K}$ be a set of $n$ disjoint copies of $\Free_{\CA}(1)$.  Then $\Free_{\CA}(n)=\coprod\class{K}$. Let $a_1,\ldots,a_n\in\Free_{\CA}(n)$ denote the $n$ free generators of $\Free_{\CA}(n)$.  Here $a_i$ will also denote  the free generator of the $i$th copy of $\Free_{\CA}(1)$ in $\class{K}$. 
 We  let $a=t^{\Free_{\CA}(n)}(a_1,\ldots,a_n)$. By assumption $\chisub{\class{K}}$ maps 
 onto $\Free_{\CA}(n)$,  
so that there exists  
$b\in \coprod\fnt U(\class{K})$ such that $\chisub{\class{K}}(b)=a$. Since $\coprod\fnt U(\class{K})=\langle\, \bigcup\fnt U(\class{K})\rangle_{\CCD}$, there exist $b_i$ in the $i$th copy of $\Free_{\CA}(1)$,  for each $i\in\{1,\ldots,n\}$ and an $n$-ary term $s$ in the language of $\CCD$ such that $s^{\coprod\fnt U(\class{K})}(b_1,\ldots,b_n)=b$. Then there exist 
unary terms
$t_1,\ldots,t_n$  in the language $\mathcal L$ such that $b_i=t^{\Free_{\CA}(1)}(a_i)$. This implies that $$t^{\Free_{\CA}(n)}(a_1,\ldots,a_n)=s^{\Free_{\CA}(n)}(t_1^{\Free_{\CA}(n)}(a_1),\ldots,t_n^{\Free_{\CA}(n)}(a_n)).$$
Since $a_1,\ldots,a_n$ are the free generators of $\Free_{\CA}(n)$ it follows that the equation $t(x_1,\ldots,x_n)\approx s(t_1(x_1),\ldots,t_n(x_n))$ is valid in $\CA$.  This proves (3).

We now prove (3) $\Rightarrow$ (4). 
 Fix $\A \in \CA$. 
Let 
$\textstyle L^{\circ}=\{\,a\in\A\mid \langle a \rangle_{\CA}\subseteq L\}
$.
Certainly   
$\bigcup\,\{\,\B\in\ope{S}(\A)\mid B\subseteq L\,\}=L^{\circ}$. 
Now let  $n \geq 1$, let $a_1,\ldots,a_n\in L^{\circ}$ and let $t$
be an 
  {$n$-ary} term in the language $\mathcal L$. 
Take $t_1,\ldots, t_n$ and ~$s$ as in the statement of~(4). 
Since $t_i^{\A}(a_i)\in\langle a_i \rangle_{\CA}\subseteq L$ for each $i\in\{1,\ldots, n\}$ and $L$ is closed under the lattice operations,  we see that  $a\in L$. This proves that $\langle a_1,\ldots,a_n\rangle _{\CA}\subseteq L$ and hence that 
$\langle a_1,\ldots,a_n \rangle_{\CA}\subseteq L^{\circ}$. We conclude that $L^{\circ}$
 is the universe of a subalgebra of $\A$ and  hence that (4) holds.

To prove (4) $\Rightarrow$ (5) we only need to observe that for every $\A_1,\A_2 \in \CA$ and 
$\w_j \in \fnt{HU}(\A_j)$ for $j=1,2$, the set $(\w_1,\w_2)^{-1}(\leq)$ is a bounded sublattice of $\fnt U(\A_1\times\A_2)$.

 Now assume that (5) holds and let $\class{K}$ be a subset of $\CA$.  
 To prove that 
(6)  holds we appeal to 
Theorem \ref{Theo:RevEng}, applied  with $\A=\coprod\class{K}$,  and also Lemma \ref{Lem:FactMap}. 
   Ler $(\vec{x_1},\w_1)$,$(\vec{x_2},\w_2)\in Y_{\coprod\class{K}}$ be such that $\Psi_{\coprod \class{K}}(\vec{x_1},\w_1)\leq \Psi_{\coprod \class{K}}(\vec{x_2},\w_2)$.
Then, for each $\B\in \class{K}$, we  have  $\w_1\circ (x_1)_{\B}\leq \w_2\circ (x_2)_{\B}$. 
Applying Theorem~\ref{Theo:RevEng} to each $\B\in\class{K}$, there exists a relation 
 $r^{\B}\in R_{\w_1,\w_2}$ such that $((x_1)_{\B}, (x_2)_{\B})\in r^{\B}$ on $\fnt D(\B)$.  By assumption $|R_{\w_1,\w_2}|\leq 1$, so
 there is a single relation $r$ such that 
$r^{\fnt D(\B)}=r$ for all~$\B\in \class{K}$. 
Therefore $r^{\fnt D(\coprod \class{K})}=\prod \{\,r^{\fnt D(\B)}\mid \B\in \class{K}\,\}$. 
Then $(\vec{x_1},\vec{x_2})\in r^{\fnt D(\coprod \class{K})}$
and  so 
$(\vec{x_1},\w_1)\preceq_{Y_{\coprod \class{K}} }(\vec{x_2},\w_2)$.  Hence (6) holds.
\end{proof}

In Theorem~\ref{Theo:GeneralS} we will reveal
that the equivalence of 
(1), (2), (3), and (4) in Theorem~\ref{Theo:EmbeddingCoprod} is independent  of the fact that $\CA$ is finitely generated.

In connection with condition~(3)  in Theorem~\ref{Theo:EmbeddingCoprod}
 we should comment on when it happens that an algebra $\A$ is always
such that   for any
$\str{L}\in\ope{S}(\fnt U(\A))$ there is at least one subalgebra of~$\A$ contained in~$L$.
It is easy to see that 
this is the case 
if and only if  $\{0^{\A},1^{\A}\}$ is the universe of a subalgebra of $\ope{S}(\A)$ for each $\A\in\CA$, or equivalently  if and only if $\Free_{\CA}(\emptyset)$ has only two elements.

Theorem~\ref{Theo:EmbeddingCoprod}
tells 
 us under what 
conditions the Priestley space of a coproduct is embeddable in the product of the Priestley spaces of the original algebras. 
When  the theorem 
can be applied, we can  combine it  with Corollary~\ref{Cor:Range} to determine the subspace of the product of Priestley spaces that is isomorphic to the Priestley space of the coproduct.

In applications it is often 
simple 
to verify or to refute  
condition (5) in Theorem~\ref{Theo:EmbeddingCoprod}.  
But we note also that, under restricted conditions on the operations in $\CA$ and the equations these satisfy, the syntactic condition~(3) may also be simple to check.
An instance is provided by \cite[Lemma~3.5]{DP87}; a slightly more general version, though still restricted to  operations which are at most unary, is given by Talukder~\cite[Lemma~6.1.2]{Ta02}.  In both cases direct proofs were given of
the existence of a unique element in a set of the form $R_{\w_1,\w_2}$
as part of the work needed to set up certain piggyback dualities.  
For use later we record here a minor variant of these
earlier results.
We stress that  the  scope of 
our syntactic condition (3) is much wider;  in particular we have been able to remove the
restriction to operations of arity at most~$1$.

\begin{lemma} \label{unique-max}  Let  $\CA$ be a $\CCD$-based quasivariety such that 
 the members of  $\CA$ take the form $\A = (A,\land,\lor,0,1,\{f_i\mid i\in I\})$, where $I$ is a finite set
and each~$f_i$ is a lattice endomorphism or a dual lattice endomorphism of $\fnt U(\A)$.  Then,  
for any  $\A\in\CA$ and any  bounded sublattice $\str{L}$ of $\fnt U(\A)$, either $\str{L}$ contains no subalgebra of~$\A$ or 
 $\{\, \B\in\ope{S}(\A)\mid B\subseteq L\,\}$ has a top element. Moreover, if each~$f_i$ preserves the bounds~$0$ and~$1$,
then $\{\B\in\ope{S}(\A)\mid B\subseteq L\}$ is non-empty and so has a top element.
\end{lemma}

Our next result,   Theorem~\ref{Theo:OntoCop}, 
 gives 
necessary and sufficient conditions for~$\fnt U$ to satisfy (E), that is, for $\chisub{\class{K}}$ to be an embedding, or equivalently, for $\iota_{\class{K}}$ to be surjective, for each $\class{K}\subseteq \CA$. 
In order to present item (4)  in the theorem, 
 we need 
 to recall some definitions and facts about quasivarieties (for 
details 
see 
\cite[Chapter~3]{Gor}). 
Let $\cat{Q}$ be a quasivariety. For each $\A\in\cat Q$,  let $\text{Con}_{\cat Q}(\A)$ denote the set of congruences 
$\theta $ on $\A$  for which 
$\A/\theta\in\cat Q$; we order $\text{Con}_{\cat Q}(\A)$ by inclusion.   
An algebra $\B\in\cat Q$ is said to be \defn{subdirectly irreducible relative to $\cat Q$} if whenever $\B$ is isomorphic to a subdirect product of algebras in $\cat Q$,
 then it is isomorphic to at least one of the components in the product.  
Let $\text{Si}(\cat Q)$ denote the class of subdirectly irreducible algebras relative to $\cat Q$.  
An algebra~$\B$ belongs to $\text{Si}(\cat Q)$  if and only if $\B\in\cat Q$ and $\Delta_{\B}\neq\bigcap(\text{Con}_{\cat Q}(\B)\setminus\{\Delta_{\B}\})$, where $\Delta_{\B}=\{\,(b,b)\mid b\in \B\,\}$.
Moreover, since $\cat Q$ is determined by a set of quasi-equations, it follows that,  if $D$ is an up-directed subset of $\text{Con}_{\cat Q}(\A)$, then $\bigcup D\in \text{Con}_{\cat Q}(\A)$.
As a consequence, 
for each $\A\in\cat Q$ and each $a,b\in \A$ such that $a\neq b$,  the set $\{\,\theta\in\text{Con}_{\cat Q}(\A)\mid (a,b)\notin\theta\,\}
$ has a maximal element.
Then every algebra in $\cat Q$ is a subdirect product of algebras in $\text{Si}(\cat Q)$. If $\cat Q=\ISP(\class{N})$, where $\class{N}$ is a finite set of finite algebras $\class{N}$, then $\text{Si}(\cat Q)\subseteq \ope{IS}(\class{N})$. 

Below we refer to an algebra  as being
\defn{non-trivial} if its has at least two elements.
We denote by $\CA_{\text{fin}}$ the class of finite algebras in~$\CA$. 

\begin{theorem}\label{Theo:OntoCop} 
 Let $\CA$ be a finitely generated
$\CCD$-based quasivariety.   
 Then the  following  statements are equivalent:
  \begin{newlist}
    \item[{\rm (1)}] $\fnt U$ satisfies 
{\upshape 
(E)}, that is, for every set  $\class{K}$ of algebras in $\CA$, the map  $\chisub{\class{K}}$ is injective;
  \item[{\rm (2)}] there exists a finite set of finite algebras $\CM$ such that  
\begin{newlist}
\item[{\rm (i)}] $\CA = \ISP(\CM)$;
\item[{\rm (ii)}] for every set $\class{K}$ of disjoint copies of algebras in $\CM$, the map  $\chisub{\class{K}}$ is injective;
\end{newlist}
\item[{\rm (3)}] there exists $\M \in \CA_{\text{\rm  fin}}$ and $\w \in
\fnt{HU}(\M) $ such that 
\begin{newlist}
\item[{\rm (i)}] $\CA = \ISP(\M)$;
\item[{\rm (ii)}] $\text{\rm (Sep)}_{\M, \w}$ holds;
\end{newlist}
  \item[{\rm (4)}] there exists $\M \in \text{\rm Si}(\CA)$ and $\w \in
\fnt{HU}(\M) $ such that 
\begin{newlist}
\item[{\rm (i)}] $\CA = \ISP(\M)$ and 
\item[{\rm (ii)}] $\text{\rm (Sep)}_{\M, \w}$ holds;
\end{newlist}
    \item[{\rm (5)}]  for every set of algebras $\class{K}\subseteq \CA$, the image of the map  $\iota_{\class{K}}$   is $\prod\fntH\fnt U(\class{K})$.
\end{newlist}
\end{theorem}
\begin{proof} 

 The equivalence of 
(1)  and (5) is a standard fact about Priestley duality and 
the implication
  (1) $\Rightarrow$ (2)  holds trivially. Since $\CA$ is finitely generated,
each $\M\in\text{\rm Si}(\CA)$ is finite. Therefore (4) $\Rightarrow$ (3). 
To complete the proof we establish that   (2) $\Rightarrow$ (3)
$\Rightarrow$ 
(5) 
and  (1)
 $\Rightarrow$ (4).

Assume (2). 
Let $\class{K}$ be a set of disjoint copies of non-trivial algebras 
from~$\CM$ which, for each $\M \in \CM$,  contains one copy $\M_{\w}$ of $\M$ for each $\w\in\fnt{HU}(\M)$. 
Observe that the assumption that $\M$ is non-trivial ensures (it is actually equivalent to)  $\fnt{HU}(\M)\neq\emptyset$.
By hypothesis,  $\iota_{ \class{K}}$ maps $\fntH\fnt U(\coprod \class{K})$ onto  
$\prod\fntH\fnt U( \class{K})$. Let $\vec{y}\in\prod\fntH\fnt U( \class{K})$ be such that the $\M_{\w}$-coordinate
 of $\vec{y}$ is~$\w$. 
Then there exists $(\vec{x},\w_0)\in Y_{\coprod \class{K}}$, where $Y_{\coprod \class{K}}$ is 
 as in Theorem~\ref{Theo:RevEng}, 
such that
$$
    \iota_{ \class{K}}(\Phi_{\coprod \class{K}}(\vec{x},\w_0))=\vec{y}.
$$
That is, 
$\w_0\circ x_{\M_{\w}}=\w$ for each non-trivial $\M\in\CM$ and each $\w\in\Omega_{\M}$.  

Let $\M_0\in \CM$ be the algebra such that $\w_0\in\Omega_{\M_0}$.
Now assume ${\M\in\CM}$ and 
let 
$a,b\in\M$  be such that $a\neq b$. Then 
there exists
 $\w\in\fnt{HU}(\M)$ such that $\w(a)\neq \w(b)$. Therefore $x_{\M_{\w}}\in\CA(\M,\M_0)$ satisfies 
\[\w_0\circ x_{\M_{\w}}(a)=\w(a)\neq\w(b)=\w_0\circ x_{\M_{\w}}(b).\]
This proves
 that $\M\in\ISP(\M_0)$ for each $\M\in\CM$ and 
hence   
$\CA=\ISP(\M_0)$. 
In addition, 
putting $\M=\M_0$ we see that ${\rm (Sep)}_{\M_0,\omega}$ holds.

We now prove that
(3) $\Rightarrow$ (5). 
Let $\CM=\{\M\}$ and $\Omega=\{\w\}$. By assumption, $\CA=\ISP(\M)$ and  $\text{(Sep)}_{\M, \w}$ holds.
Then we can
 apply Theorem~\ref{genpig}.  Now, let~$\class{K} \subseteq \CA$. 
 By Lemma~\ref{jointsurj},
$\{\w\}=\Lambda_{\B}(f)$ for each $\B\in\class{K}$ and each $f\in\fntH\fnt U(\B)$. By  Corollary~\ref{Cor:Range}, 
\[
\textstyle
\iota_{\class{K}}\bigl(
\{\,\vec{y}\in\prod\fntH\fnt U(\class{K})\mid \bigcap\{\,\Lambda_{\B}(y_{\B})\mid \B\in \class{K}\,\}\neq\emptyset\,\}\bigr)  =\prod\fntH\fnt U(\class{K}).
\]

Finally we prove  (1) $\Rightarrow$ (4). Let $\CM$ be a finite set of algebras such that ${\CA=\ISP(\CM)}$.
 Let $\class{N}=\text{Si}(\CA)\cap \ope{S}(\CM)$.  Clearly
$\class{N}$ is
 a finite set of finite algebras. Since $\CA=\ISP(\CM)$ we have, 
on the one hand, $\CA=\ISP(\class{N})$. By assumption
$\fnt U$ satisfies condition (E), and in particular for every set $\class{K}$ of disjoint copies of algebras in $\class{N}$, the map  $\chisub{\class{K}}$ is injective. Now the proof of 
the implication 
(2) $\Rightarrow$ (3) above, applied 
to  $\class{N}$, implies that
 there exists $\M\in\class{N}\subseteq{\rm Si}(\CA)$ and $\w\in\fnt{HU}(\M)$
 satisfying (i) and (ii) in (4).
\end{proof}

We now 
present 
our main theorem, as advertised in Section~\ref{Sec_Intro}. 
This amalgamates 
the results from Theorems~\ref{Theo:EmbeddingCoprod} and  \ref{Theo:OntoCop}.

\begin{theorem}\label{MainTheo}  {\upshape 
(Coproduct Preservation Theorem)} 
Let~$\CA$ be a 
finitely generated $\CCD$-based quasivariety and $\fnt{U} \colon \CA \to \CCD $ the associated forgetful functor. 
 Then the following statements are equivalent:
\begin{newlist}
\item[{\rm (A)}] $\fnt U\colon\CA\to\CCD$ preserves coproducts; 
\item[{\rm (B)}] 
the following conditions hold:
\begin{newlist}
\item[{\rm (i)}] there exist a finite algebra 
$\M\in\CA$ and 
$\w \in \CCD(\fnt{U}(\M),\two)$ 
 such that
  $\CA=\ISP(\M)$ and $\text{\rm (Sep)}_{\M,w}$ holds
{\upshape(}that is,  for all $a,b\in \M$, if $\w (u(a)) =\w (u(b))$ for each $u\in\CA(\M,\M)$, then $a=b${\upshape)},
\item[{\rm (ii)}] for every $\A\in\CA$ and every  bounded sublattice 
$\str{L}$ of $\fnt U(\A)$ 
 the subposet 
$$
\{\, \C\in\ope{S}(\A)\mid C\subseteq L\, \} 
$$
of $\ope{S}(\A)$ is empty or has a top element;
\end{newlist}

\item[{\rm (C)}] there exists a finite algebra $\M\in\CA$ which is subdirectly irreducible  relative to $\CA$ and 
there exists
$\w\in \CCD(\fnt U(\M),\two)$ such that
\begin{newlist}
\item[{\rm (i)}] 
every algebra which is subdirectly irreducible relative to $\CA$
belongs to $\ope{IS}(\M)$,
\item[{\rm (ii)}] for every $a,b\in \M$, if $\w (u(a)) =\w (u(b))$ for each $u\in\CA(\M,\M)$, then $a=b$,
\item[{\rm (iii)}] the subposet 
$$
\{\,\C\in\ope{S}(\M^2)\mid \forall (c_1,c_2)\in\C, \w(c_1)\leq\w(c_2) \,\}
$$
of $\ope{S}(\M^2)$
has a 
top 
element.

\end{newlist}
\end{newlist}
\end{theorem}
\begin{proof}
The functor $\fnt U$ preserves coproducts if and only if it satisfies conditions~(E) and (S).
Condition (B)(i) is exactly item (3) in Theorem~\ref{Theo:OntoCop}. Then~$\fnt U$ satisfies (E) if and only if  (B)(i) holds. 
Likewise,  item (4) in Theorem~\ref{Theo:EmbeddingCoprod} tells us that $\fnt U$ satisfies (S) if and only if 
(B)(ii) holds. 

Assume that (A) holds. 
Observe that it is a 
straightforward consequence of 
Theorem~\ref{Theo:OntoCop} that 
(C)(i) and (C)(ii) hold. 
Also Theorem~\ref{Theo:EmbeddingCoprod}  implies that  (C)(iii) holds.
We have  proved that (A) $\Rightarrow$ (C).

To prove (C) $\Rightarrow$ (A), we proceed as follows.  
Let us fix $\CM=\{\M\}$ and $\Omega=\{\w\}$. Condition (C)(i) 
gives 
 $\CA=\ISP(\M)$, and condition (C)(ii) implies that ${\rm (Sep)}_{\M,\w}$ holds. Then, by Theorem~\ref{Theo:OntoCop}, conditions (C)(i)--(ii) imply that~$\fnt U$ satisfies (E). Now observe that (C)(iii) is equivalent to 
the assertion  that $|R_{\w,\w}|=1$. Since $\CA=\ISP(\M)$, and ${\rm (Sep)}_{\M,\w}$ holds,  Theorem~\ref{Theo:EmbeddingCoprod} implies that $\fnt U$ satisfies (S).
\end{proof}

The equivalence of  (A) and (B) 
spells out the characteristic  
properties  of the class $\CA$ that hold if and only if 
 $\fnt U$ preserves coproducts.
More precisely, it tells us how $\CA$ interacts with the subclass 
$\fnt{U}(\CA)$  of $\CCD$. 
The equivalence of~(A) and (C) brings  benefits of a different kind. 
Assume that 
 we are presented with a finite family of finite algebras $\CM$ such that $\CA=\ISP(\CM)$ is $\CCD$-based.  Then
condition~$(C)$  
enables us to assert that the decision   
problem ``$\fnt U$ preserves coproducts'' is decidable. 
To see this, 
first note that
${\rm Si}(\CA)\subseteq\ope{IS}(\CM)$. Then~(C) implies that we only need to check (i)--(iii) on the finite set of pairs $(\M,\w)$ where $\M\in\ope{S}(\CM)$ and $\w\in \fnt{HU}(\M)$. 
Having $\CM$ to hand,  to check (C)(i) on a particular pair $(\M,\w)$ amounts  to proving that  $\str{N}\in\ISP(\M)$ for each $\str{N}\in\CM$. This  is decidable since each $\str{N}\in\CM$ is finite. This, together with the fact that Items (C)(ii)--(iii) are clearly decidable for each pair $(\M,\w)$, 
confirms 
our decidability claim concerning preservation of coproducts by~$\fnt U$.
We remark that the restriction to subdirectly irreducible algebras 
is not itself pertinent to the  decidability question.  What is essential is that
$\CA$ should be generated by a single finite algebra~$\M$.
Moreover, since $\text{Si}(\ISP(\CM))\subseteq\ope{IS}(\CM)$, the finite
algebra $\M$  can  be chosen  to be a member of  $\CM$. 
There are two main reasons for~(C) to be formulated as it is. 
Firstly, it reflects the way the required property  would customarily be verified in practice. Secondly,
maximal subdirectly irreducible algebras of $\CA$ are absolute retracts within $\CA$,
and hence, if $\M\in\text{Si}(\CA)$ satisfies (C)(i), then $\M$ is a retract of any $\A\in\CA$ satisfying  (C)(i). Therefore, performing a full check of (C)(ii)--(iii) is faster in a subdirectly irreducible algebra satisfying (C)(i) than in a non-subdirectly irreducible one.

Up to this point, we have used natural dualities to study the
behaviour 
 of coproducts in  finitely generated $\CCD$-based quasivarieties. 
 Theorems~\ref{Theo:EmbeddingCoprod} and \ref{Theo:OntoCop} can  also be used in the reverse direction.  
By this we mean that 
the type  of natural duality that we can obtain for a given  class of algebras is governed  by the 
properties of coproducts  in that class. 
Most of the information we require  in order to make this  assertion precise 
can be extracted from what we have proved already, but before we can reveal the full 
picture we need to recall some definitions.  

A coproduct in a class $\CA$ of a set $\class{K}$ of non-trivial 
algebras in  $\CA$ is 
 a \defn{free product}
if the universal co-cone $\{\epsilon_{\A}\colon \A\to\coprod\class{K}\mid \A\in\class{K}\}$ is such that each~$\epsilon_{\A}$ is an embedding. When this is the case  for every such $\class{K}$, then $\CA$ is said \defn{to admit 
free products}. 
A class of algebras $\CA$ is said to have the \emph{embedding property} if for each set of non-trivial algebras $\class{K}$ in $\CA$ there exists an algebra
into which  all the algebras in $\class{K}$ can be embedded. In the case that a class of algebras admits coproducts, the two properties---%
admitting free products and the embedding property---are 
known to be 
equivalent.

Observe that $\CCD$ admits free products. Therefore, if $\CA$ is a class of $\CCD$-based algebras
 such that $\fnt U$ satisfies (E), then $\CA$ also admits free products.
 In Theorem~\ref{Theo:OntoCop}
we 
 proved that,  given a finitely generated $\CCD$-based quasivariety $\CA$, the functor $\fnt U$ satisfies (E) if and only if there exists a finite algebra generating $\CA$, together 
 with a particular
lattice homomorphism 
from (the reduct of) this algebra
into $\two$. This shows 
that there is a
deep connection between free products and condition (E), since
a quasivariety admits free products (or, equivalently, has the embedding property) if and only if it is generated by a single algebra
 \cite[Proposition~2.1.19]{Gor}.
By modifying part of the argument used to prove  Theorem~\ref{Theo:OntoCop} we can 
 present a different proof of the equivalence between free products and single generation for the case of $\CCD$-based quasivarieties
which 
 reveals more overtly 
the  connection between piggyback dualities and free products.

\begin{theorem}\label{Theo:FreeProd}
Let $\CA$ be a finitely generated $\CCD$-based quasivariety. Then $\CA$ admits free products if and only if there exists a {\upshape(}finite{\upshape)} algebra $\M\in\CA$ such that $\CA=\ISP(\M)$. 
\end{theorem}
\begin{proof}
For the
forward 
implication,  let $\CM$ be a finite set of finite algebras such that $\CA=\ISP(\CM)$. Since $\CA$ admits free products, $\CM\subseteq\mathbb{IS}(\coprod \CM)$ and hence  $\CA=\ISP(\coprod \CM)$. Moreover, since $\CA$ is locally finite,
$\coprod \CM$ is finite.

Assume now that $\CA=\ISP(\M)$ for some algebra $\M$. There is no loss of generality 
in 
assuming that $\M$ is finite. 
Now $\text{Sep}_{\M,\Omega}$ holds, where 
 $\Omega=\fntH\fnt U(\M)$.
Let $\class{K}$ be a set of non-trivial algebras in $\CA$ and   
let $\{\epsilon_{\B}\mid \B\in \class{K}\}$
 be the universal co-cone that determines the coproduct $\coprod \class{K}$.
 Proving  that each $\epsilon_{\B}$ is an embedding is equivalent to proving  that the image of $\fntH\fnt U(\epsilon)$ is
$\fntH\fnt U(\B)$.
Let $h\in \fntH\fnt U(\B)$.   By Theorem~\ref{Theo:RevEng} 
(referring back to Lemma~\ref{jointsurj})
there exists $(x,\w)\in Y_{\B}$ such that ${\Phi_{\B}(x,\w)=h}$. Since $\CM=\{\M\}$ and $\Omega=\fntH\fnt U(\M)$, we have
$x\in\CA(\B,\M)$ and $\w\in\fnt{HU}(\M)$. 
 Let $x_{\B}=x$ and, for each $\A\in\class{K}\setminus\{\B\}$,  let $x_{\A}$ be an arbitrary element of $\CA(\A,\M)$
(since $\A$   is non-trivial, this set is non-empty).
Then $(\vec{x},\w)\in Y_{\coprod\class{K}}$. From  Fig.~\ref{Fig:Factor},  
\[
\fntH\fnt U(\epsilon_{\B})(\Psi_{\coprod\class{K}}(\vec{x},\w))=\Phi_{\B}\circ\xi_{\B}(\vec{x},\w)=\w\circ x_{\B}
=\w\circ x=h.
\] 
It follows that the image of  $\fntH\fnt U(\epsilon_{\B})$ is 
$\fntH\fnt U(\B)$.
\end{proof}

 The following theorem collects together the connections between the properties of coproducts in a finitely generated $\CCD$-based quasivariety and the type of natural duality that it admits. Its proof is a direct application of Theorems~\ref{genpig}, \ref{Theo:EmbeddingCoprod},  \ref{Theo:OntoCop} and \ref{Theo:FreeProd} and Corollary~\ref{Cor-to-RevEng}.

\begin{theorem}\label{Theo:CoproToNatDual}
Let $\CA$ be a finitely generated $\CCD$-based quasivariety and let $\fnt U\colon\CA\to\CCD$ be the corresponding forgetful functor. 
\begin{newlist}
\item[{\rm (i)}] $\CA$ admits free products if and only if it admits a 
single-sorted natural duality.
\item[{\rm (ii)}] $\fnt U$ satisfies
{\upshape  
(E)} if and only if 
$\CA$
 admits a simple piggyback duality. 
\item[{\rm (iii)}] $\fnt U$ satisfies 
{\upshape 
(S)}  if and only if 
$\CA$
 admits a piggyback duality,
which may be single-sorted or multisorted, 
 but which is such that $|R_{\w_1,\w_2}|\leq 1$ for each pair of carrier
maps  
 $\w_1$, $\w_2$.  
Moreover, if any piggyback duality for 
$\CA$ has this feature  then all such dualities do.  
\item[{\rm (iv)}] $\fnt U$ preserves coproducts if and only if $\CA$ admits a simple piggyback duality that is also a \DP-based duality.
\end{newlist}
\end{theorem}

\section{Stability properties of term reducts}\label{Sec:CopSub}
Until now we have focused on finitely generated $\CCD$-based
quasivarieties, and made full use of the duality theory that is thereby available.  However it is striking that in our surjectivity theorem
(Theorem~\ref{Theo:EmbeddingCoprod}) 
certain of the equivalent conditions 
would be meaningful in a wider setting, 
 We  accordingly now seek to ascertain how far  the results of the  previous section depend 
critically 
on the assumptions we made about $\CA$.
There are two respects in which we may extend the setting.  
First of all we demanded that our quasivarieties be finitely generated
in order to exploit techniques from natural duality theory; this is not a necessary requirement for coproducts to be available.  Secondly, we may
seek to remove the restriction that our base variety be $\CCD$.

In what follows it will be appropriate  to consider the following generalisation of the notion of a quasivariety. 
A class of algebras $\CA$ is called a \defn{prevariety} if $\ope{ISP}(\CA)=\CA$. 
For each algebra $\A$ in a given prevariety $\CA$  the set of congruences $\Cong_{\CA}(\A)=\{\,\theta\in\Cong(\A)\mid \A{/}{\theta}\in\CA\,\}$ is a complete lattice for the usual inclusion order.
We note that,  
as a consequence,  any 
prevariety 
$\CA$ admits free objects and coproducts.

We shall extend our earlier definition so as to encompass
prevarieties $\CA$ from which there exists  an appropriate  natural forgetful functor  
$\fnt U_{\CA, \cat C}$ to  a specified class $\cat C$ of algebras. 
The properties of the functor $\fnt U_{\CA,\cat C}$ with respect to coproducts will depend on both  
$\CA$ and $\cat C$.  We remain interested principally
in the case that $\cat C=\CCD$ and  make no claim to 
give an exhaustive analysis of  
the properties of $\fnt U_{\CA,\cat C}$ with respect to coproducts.   
However we do take our  study far enough to reveal which properties of the forgetful functor  are preserved when we restrict to subprevarieties of $\CA$.  This enables us to enlarge the range of classes of algebras 
in which coproducts can be identified.

 Let 
$\CA$ be a class of algebras with language $\mathcal{L}$. For each set $\mathcal T$ of $\mathcal{L}$-terms, the \defn{$\mathcal T$-reduct} of an algebra $\A\in\CA$ is defined to be the algebra 
$(A,\{t^{\A}\mid t\in\mathcal T\}),$
where $A$ is the universe of $\A$. The algebra $\fnt U_{\CA,\mathcal T}(\A)=(A,\{t^{\A}\mid t\in\mathcal T\})$ is called a \defn{ term-reduct} of $\A$.
This notion of term-reduct encompasses the usual notion of reduct in the following way.
Let $\mathcal{L}'\subseteq \mathcal{L}$ and, for an  $n$-ary   
operation $f\in\mathcal{L}'$,  
let $\mathcal T=\{t_f\mid f\in\mathcal{L}'\}$,
 where for an $n$-ary operation $f$ the term $t_f$ is $f(x_1,\ldots,x_n)$ for variables $x_1,\ldots,x_n$. Then 
$$
\fnt U_{\CA,\mathcal T}(\A)=(\A,\{t_f^{\A}\mid f\in\mathcal{L}\})=(\A,\{f^{\A}\mid f\in\mathcal{L}\}).
$$ 
Since each $\mathcal{L}$-term has a well-defined arity, we can consider $\mathcal T$ as a language. If $\cat C$ is a class of $\mathcal T$-algebras containing $\fnt U_{\CA,\mathcal T}(\A)$ for each $\A \in\CA$, we shall write~$\fnt U_{\CA,\cat C}(\A)$ when we wish to highlight that we are considering the algebra $\fnt U_{\CA,\mathcal T}(\A)$ as an algebra in $\cat C$. In this case we say that $\CA$ is \defn{ $\cat C$-based}. Observe that if $\cat C$ is a variety and a class of algebras $\CA$ is 
$\cat C$-based, then $\ope{V}(\CA)$, the variety generated by $\CA$, is also $\cat C$-based.  

Assume that a given class $\CA$ of algebras is $\cat C$-based.  We wish to define $\fnt U_{\CA,\cat C}$ on morphisms so as to make it a 
functor from~$\CA$ into the class~$\cat C$.
For every $n$-ary $\mathcal{L}$-term $t$ and every homomorphism $h\colon \A\to\B$, we have $h(t^{\A}(a_1,\ldots,a_n))=t^{\B}(h(a_1),\ldots,h(a_n))$ for each $a_1,\ldots,a_n\in\A$. 
Then 
the assignment $\fnt U_{\CA,\cat C}(h)=h$  does indeed 
make $\fnt U_{\CA,\cat C}$ 
into a functor. 
In the case that $\CA$ is $\CCD$-based, we omit the subscripts 
and  write $\fnt{U}$ instead of
 $\fnt U_{\CA,\CCD}$.

Let  $\class{K}$ be a set of algebras in $\CA$ and let $\coprod_{\CA} \class{K}$ be its coproduct in $\CA$.
Let $\{\,\epsilon_{\B}\colon \B\to\coprod_{\CA} \class{K}\mid \B\in \class{K}\,\}$
 be the universal co-cone that determines the coproduct $\coprod_{\CA} \class{K}$ up to isomorphism. There exists a unique map 
\begin{equation}\tag{Chi}\label{Eq:DefChi}
  \textstyle\chisub{\class{K}}\colon \coprod_{\cat C}\fnt U_{\CA,\cat C}(\class{K})\to \fnt U_{\CA,\cat C}(\coprod_{\CA}\class{K})
\end{equation}
such that $\chisub{\class{K}}\, \circ\, \varepsilon_{\B}=\fnt U_{\CA,\cat C}(\epsilon_{\B})$ for each $\B\in\class{K}$, where the family of homomorphisms
$  \{\,\varepsilon_{\B}\colon \fnt U_{\CA,\cat C}(\B)\to \coprod_{\cat C}\fnt U_{\CA,\cat C}(\class{K})\mid \B\in\class{K}\,\}$ is the universal co-cone in~$\cat C$.
In the previous section we investigated
when the functor  
$\fnt U$, for a fixed $\CCD$-based quasivariety $\CA$,
 preserves coproducts, determining when $\chisub{\class{K}}$ is injective and/or surjective.  
We extend our earlier usage and 
say that the functor~$\fnt U_{\CA,\cat C}$ satisfies
\begin{newlist}
\item[(S)] if for each set $\class{K}\subseteq\CA$, the map $\chisub{\class{K}}$ is surjective;
\item[(E)] if for each set $\class{K}\subseteq\CA$, the map $\chisub{\class{K}}$ is injective.
\end{newlist}

A remark is in order here. The only result of this section in which we study the properties of condition (E) is Theorem~\ref{Theo:RetractCons}. The more amenable properties of condition (S) are mainly due to the fact that we can relate it to a property of the free objects in $\CA$ as in Theorem~\ref{Theo:EmbeddingCoprod}(2), while condition (E) depends on the set of generating algebras, or the set ${\rm Si} (\A)$,
 as can be 
seen 
from Theorem~\ref{Theo:OntoCop}.

The first result of this section 
generalises
Theorem~\ref{Theo:EmbeddingCoprod}(1)--(4).
\begin{theorem}\label{Theo:GeneralS} Let
$\CA$ and $\cat C$ be prevarieties such that $\CA$ is $\cat C$-based.
Then
  the following statements are equivalent:
  \begin{newlist}
    \item[{\rm (1)}] $\fnt U_{\CA,\cat C}$ satisfies {\rm (S)};
    \item[{\rm (2)}] for every finite set   $\class{K}$ of disjoint copies of 
$\Free_{\CA}(1)$, 
the image of $\chisub{\class{ K}}$ is 
$\fnt U_{\CA,\cat C}(\coprod_{\CA}\class{K})$;
\item[{\rm (3)}] for every $n\geq 1$ and every  $n$-ary term $t$ in the language of $\CA$ there exist  unary terms $t_1,\ldots,t_n$ in the language of $\CA$ and  an $n$-ary term $s$ in the language of $\cat C$ such that 
$$
t^{\A}(a_1,\ldots,a_n)=s^{\A}(t_1^{\A}(a_1),\ldots,t_n^{\A}(a_n)),
$$
 for every $\A\in\CA$ and every $a_1,\ldots,a_n\in\A$; 
\item[{\rm (4)}] for every $\A\in\CA$ and $\str{B}\in\ope{S}(\fnt U_{\CA,\cat C}(\A))$,  the set  
$\{\str{C}\in\ope{S}(\A)\mid C\subseteq B\}$ is empty or has a top element.
\end{newlist}
\end{theorem}
\begin{proof}  
The implication (1) $\Rightarrow$ (2) is straightforward. To prove the implications (2) $\Rightarrow$ (3) $\Rightarrow $ (4) it is enough to replace $\CCD$ by $\cat C$ in the proofs of the corresponding implications in  Theorem~\ref{Theo:EmbeddingCoprod}, 
since the latter do not depend on the fact that $\CA$ is finitely generated.

 We now  prove (4) $\Rightarrow $ (1). 
Let us consider 
 a set  $\class{K}$ of algebras 
in $\CA$ and let 
$\{\,\epsilon_{\A}\colon \A\to\coprod_{\CA} \class{K}\mid \A\in \class{K}\}$ be the universal co-cone that determines~$\coprod_{\CA} \class{K}$.   
Now let $\B=\langle \,\bigcup \{\,\epsilon_{\A}(\A)\mid\A\in\class{K}\,\}\rangle_{\cat C}\subseteq \coprod_{\CA} \class{K}$.  Then $\B\in \ope{S}(\fnt U_{\CA,\cat C}(\coprod\class{K}))$ and $\epsilon_{\A}(\A)\subseteq \B$ for each $\A\in\class{K}$. 
If $\{\,\varepsilon_{\A}\colon \A\to\coprod_{\CA} \fnt U_{\CA,\cat C}(\class{K})\mid \A\in \class{K}\}$ denotes  the universal co-cone that determines $\coprod_{\cat C} \fnt U_{\CA,\cat C}(\class{K})$, it follows that 
$\textstyle\coprod_{\cat  C}\fnt U_{\CA,\cat C}(\class{K})=\langle\, \bigcup \{\,\varepsilon_{\A}(\A)\mid\A\in\class{K}\,\}\rangle_{\cat C}$.
 By (\ref{Eq:DefChi}) we have  
${\chisub{\class{K}}(\coprod\fnt U_{\CA,\cat C}(\class{K}))=B}$.
By assumption, $\{\str{C}\in\ope{S}(\coprod\class{K})\mid C\subseteq B\}$ has a top element, $\str{D}$ say. Then $\epsilon_{\A}(\A)\subseteq \str{D}$ for each $\A\in\class{K}$. Finally, since $\langle\, \bigcup \{\,\epsilon_{\A}(\A)\mid\A\in\class{K}\,\}\rangle_{\CA}=\coprod_{\CA}\class{K}$, we have that $\str{D}=\coprod_{\CA}\class{K}$ and that the universes $D$ and $B$ coincide. We conclude that $\chisub{\class{K}}(\coprod\fnt U_{\CA,\cat C}(\class{K}))=\coprod_{\CA}\class{K}$.
\end{proof}

\begin{corollary}\label{Cor:USextend}
 Let
$\CA$ and $\cat C$ be prevarieties 
such that 
the variety $\ope{V}(\CA)$ generated by $\CA$ is $\cat C$-based. Then the following statements are equivalent:
\begin{newlist}
\item[{\rm (1)}] $\fnt U_{\CA,\cat C}$ satisfies {\rm (S)};
\item[{\rm (2)}] $\fnt U_{\cat B,\cat C}$ satisfies {\rm (S)} for each prevariety $\cat B$ such that $\ope{ISP}(\Free_{\CA}(\aleph_0))
\subseteq \cat B\subseteq \ope{V}(\CA)$.
\end{newlist}
\end{corollary}
\begin{proof}
The result follows directly from Theorem~\ref{Theo:GeneralS} and the observation that 
$\Free_{\ope{V}(\CA)}(\lambda)=\Free_{\CA}(\lambda)=\Free_{\ope{V}(\cat B)}(\lambda)$ for any cardinal $\lambda$ and each 
$\cat B$ such that 
$\ope{ISP}(\Free_{\CA}(\aleph_0))\subseteq \cat B\subseteq \ope{V}(\CA)$.
\end{proof}

We observe that  we can relax the 
assumption of finite generation in Theorem~\ref{Theo:EmbeddingCoprod} 
 to local finiteness and still have a condition 
similar 
 to condition (5).

\begin{theorem}
Let $\CA$ be a locally finite $\CCD$-based quasivariety. Then the following statements are equivalent:
\begin{newlist}
\item[{\rm (1)}] $\fnt U$ satisfies {\rm (S)};
\item[{\rm (2)}] for each $n\geq 1$ and each $\w_1,\w_2\in\fnt{HU}(\Free_{\CA}(n))$, the set $$\{\, \C\in\ope{S}(\Free_{\CA}(n)^2)\mid C\subseteq (\w_1,\w_2)^{-1}(\leqslant)\, \} $$ is 
either  
empty or has a top element.
\end{newlist}
\end{theorem}
\begin{proof}  

Here we shall need to consider other quasivarieties
besides $\CA$ itself and their forgetful functors to~$\CCD$.  We 
add subscripts to indicate the domains of the functors.

The implication (1) $\Rightarrow $ (2) follows from Theorem~\ref{Theo:GeneralS} and the observation that $(\w_1,\w_2)^{-1}(\leqslant)$ is a sublattice of $\fnt U(\Free_{\CA}(n))$.

Let us 
now prove (2) $\Rightarrow$ (1).   
Fix  $n\geq 1$ and let $\cat B=\ope{ISP}(\Free_{\CA}(n))$.
Let $\CM=\{\Free_{\CA}(n)\}$ and $\Omega=\fnt{HU}_{\CA}(\Free_{\CA}(n))$. By assumption   $|R_{\w_1,\w_2}|\leq 1$  for each $\w_1,\w_2\in\Omega$.  By Theorem~\ref{Theo:EmbeddingCoprod}, the functor $\fnt U_{\cat B}$ satisfies~(S).
Now 
 let $\class{K}$ be a set with $n$ disjoint copies of $\Free_{\CA}(1)$. Then ${\coprod_{\CA}\class{K}=\Free_{\CA}(n)=\coprod_{\cat B}\class{K}}$. It follows that  the map $\chisub{\class{K}}\colon \coprod_{\CCD}\fnt U_{\CA}(\class{K})\to\fnt U_{\CA}(\coprod_{\CA}\class{K})$
 coincides with the map  
${\chisub{\class{K}}'\colon \coprod_{\CCD}\fnt U_{\cat B}(\class{K})\to\fnt U_{\cat B}(\coprod_{\cat B}\class{K})}$, as defined in \eqref{Eq:DefChi}. The latter 
map is 
surjective by hypothesis. Since this argument is valid for every $n\geq 1$,  Theorem~\ref{Theo:GeneralS} proves that $\fnt U_{\CA}$ satisfies~(S).
\end{proof}

We now wish to pursue the idea that a given prevariety 
$\cat Q$ may sometimes be contained  in another prevariety~$\cat Q'$
which has better properties as regards coproducts  and that this may 
assist us in describing coproducts in~$\cat Q$.  
Let $\cat{Q}$ and $\cat{Q}'$ be prevarieties such that $\cat{Q}\subseteq\cat{Q}'$. 
The category $\cat{Q}$ can be viewed  as a reflective subcategory of $\cat{Q}'$.
 Indeed, for each $\A\in\cat{Q}$, the set of congruences $\theta$ of $\A$ such that $\A/\theta\in\cat{Q}'$ 
has a bottom element;  we  denote this by $\theta_{\cat{Q}}(A)$. 
Moreover,  
given a homomorphism
$h\colon\A\to\B$ there is a 
unique homomorphism  $h'\colon\A/\theta_{\cat{Q}}(A)\to\B/\theta_{\cat{Q}}(A)$
that makes the diagram in Fig.~\ref{Fig:Reflector} commute. The 
assignment  $\A\mapsto\A/\theta_{\cat{Q}}(A)$ and 
$h\mapsto h'$  
is then a well-defined functor $\fnt{R}_{\cat{Q}'}\colon\cat{Q}\to\cat{Q}'$.
It follows that $\fnt{R}_{\cat{Q}'}$ is left adjoint to the inclusion functor from~$\cat{Q}'$ into~$\cat{Q}$. (See \cite[Theorem~2.1.8]{Gor} or \cite[p.~235]{Mal1970} for the existence of the minimal congruence and \cite[Corollary~4.22]{Man1976} for its categorical properties.)
Therefore $\cat{Q}'$ is a reflective subcategory of $\cat{Q}$. 
The following proposition now tells us that we can obtain coproducts
in $\cat Q$ provided we have  descriptions of coproducts in~$\cat Q'$
and of the congruence $\theta_{\cat{Q}}$.
This tactic was  employed for example in  \cite{Ci79,CF79} 
and we  shall 
use it several times in Section~\ref{Sec_Ex}.

\begin{figure} [ht]
\begin{center}
\begin{tikzpicture} 
[auto,
 text depth=0.25ex,
 move up/.style=   {transform canvas={yshift=2.5pt}},
 move down/.style= {transform canvas={yshift=-2pt}},
 move left/.style= {transform canvas={xshift=-2.5pt}},
 move right/.style={transform canvas={xshift=2.5pt}}] 
\matrix[row sep= .65cm, column sep= .9cm] 
{ \node (A) {$\A$};  & \node (B) {$\B$};\\
\node (A') {$\A/\theta_{\cat{Q}(\A)}$}; & \node (B') {$\B/\theta_{\cat{Q}(\B)}$};\\};
\draw [->] (A) to node  {$h$} (B);
\draw [->] (A') to node [swap]  {$h'$} (B');  
\draw [->>] (A) to node {}(A');
\draw [->>] (B) to node {}(B');
\end{tikzpicture}
\end{center}\caption{Functoriality of the reflector $\fnt R_{\cat Q'}$}\label{Fig:Reflector}
\end{figure} 

\begin{proposition} \label{Prop:Reflect} Let $\cat{Q}$ and $\cat{Q'}$ be prevarieties such that $\cat{Q}'\subseteq \cat{Q}$.  
Let $\class{K}$ be a set of algebras in $\cat{Q}'$.  Then
$$
\textstyle\coprod_{\cat{Q}'}\class{K}\cong\fnt{R}_{\cat{Q}'}
\bigl(\coprod_{\cat{Q}}\class{K}\bigr)=\bigl(\coprod_{\cat{Q}}\class{K}\bigr)/(\theta_{\cat{Q}'}(\coprod_{\cat{Q}}\class{K})).
$$
\end{proposition}

\begin{figure}  %[ht]
\begin{center}
\begin{tikzpicture} 
[auto,
 text depth=0.25ex,
 move up/.style=   {transform canvas={yshift=2.5pt}},
 move down/.style= {transform canvas={yshift=-2pt}},
 move left/.style= {transform canvas={xshift=-2.5pt}},
 move right/.style={transform canvas={xshift=2.5pt}}] 
\matrix[row sep= .9cm, column sep= .9cm]
{ \node (CUK) {$\coprod_{\cat C}\fnt U_{\CA,\cat C}(\class{K})$ };  
& \node (UCK) {$U_{\CA,\cat C}(\coprod_{\CA}\class{K})$};\\
\node (CUK1) {$\coprod_{\cat C}\fnt U_{\CA',\cat C}(\class{K})$ };  
& \node (UCK1) {$U_{\CA',\cat C}(\coprod_{\CA'}\class{K})$};\\
 };
\draw [->] (CUK) to node  {$\chisub{\class{K}}$} (UCK);
\draw [->] (CUK1) to node [swap] {$\chisub{\class{K}}'$} (UCK1);
\draw [<->] (CUK) to node [swap] {${\rm Id}$} (CUK1);
\draw [->>] (UCK) to node {$\fnt U_{\CA,\cat C}(\rho)$} (UCK1);
\end{tikzpicture}
\end{center}\caption{The proof of Theorem~\ref{Theo:S-Hereditary}}\label{Fig:Chis}
\end{figure} 

Now we assemble some consequences of Proposition~\ref{Prop:Reflect}. 

\begin{theorem}\label{Theo:S-Hereditary}
Let
$\CA$ and $\cat C$ be prevarieties such that $\CA$ is $\cat C$-based. Then $\fnt U_{\CA,\cat C}$ satisfies {\rm (S)} if and only if $\fnt U_{\CA',\cat C}$ satisfies {\rm (S)}  for each prevariety $\CA'\subseteq\CA$. 
\end{theorem}

\begin{proof}
For the non-trivial implication, let $\CA'$ be a prevariety contained in $\CA$. Since $\CA$ is $\cat C$-based,  $\CA'$ is also $\cat C$-based, and we can consider
the functor 
$\fnt U_{\CA',\cat C}\colon\CA\to \cat C$.
Let $\class{K}\subseteq \CA'$ be a set of algebras. We will use $'$ to distinguish  $\chisub{\class{K}}\colon\coprod_{\cat C}\fnt U_{\CA,\cat C}(\class{K})\to \fnt U_{\CA,\cat C}(\coprod_{\CA}\class{K})$ from $\chisub{\class{K}}'\colon\coprod_{\cat C}\fnt U_{\CA',\cat C}(\class{K})\to \fnt U_{\CA',\cat C}(\coprod_{\CA'}\class{K})$. By Proposition~\ref{Prop:Reflect} we have 
 $\coprod_{\CA'}\class{K}\cong(\coprod_{\CA}\class{K})/(\theta_{\CA'}(\coprod_{\CA}\class{K}))$. Let $\rho\colon\coprod_{\CA}\class{K}\to \coprod_{\CA'}\class{K}$ denote the quotient map. Since $\fnt U_{\CA',\cat C}(\coprod_{\CA'}\class{K})= 
\fnt U_{\CA,\cat C}(\coprod_{\CA'}\class{K})$, 
the diagram in Fig.~\ref{Fig:Chis} commutes. 
Thus, if $\chisub{\class{K}}$ is surjective,  so is $\chisub{\class{K}}'$.
\end{proof}

\begin{theorem}\label{Theo:RetractCons}
Let
$\CA$ and $\cat C$ be prevarieties such that $\CA$ is $\cat C$-based. If $\cat C'$ is a prevariety such that $\cat C\subseteq \cat C'$, then the following statements hold:
\begin{newlist}
\item[{\rm (i)}] if $\fnt U_{\CA,\cat C}$ satisfies {\rm (S)} then $\fnt U_{\CA,\cat C'}$ satisfies {\rm (S)};
\item[{\rm (ii)}] if $\fnt U_{\CA,\cat C'}$ satisfies {\rm (E)} then $\fnt U_{\CA,\cat C}$ satisfies {\rm (E)}.
\end{newlist}
\end{theorem}
\begin{proof}
Let $\class{K}\subseteq \CA$ be a set of algebras. Now we will use $'$ to distinguish $\chisub{\class{K}}\colon\coprod_{\cat C}\fnt U_{\CA,\cat C}(\class{K})\to \fnt U_{\CA,\cat C}(\coprod_{\CA}\class{K})$ from $\chisub{\class{K}}'\colon\coprod_{\cat C'}\fnt U_{\CA,\cat C'}(\class{K})\to \fnt U_{\CA,\cat C'}(\coprod_{\CA}\class{K})$. By Proposition~\ref{Prop:Reflect},   
$$\textstyle\coprod_{\cat C}\fnt U_{\CA,\cat C}(\class{K})\cong(\coprod_{\cat C'}\fnt U_{\CA,\cat C'}(\class{K}))/(\theta_{\cat C}(\coprod_{\cat C'}\fnt U_{\CA,\cat C'}(\class{K})).
$$
 Let $\rho\colon\coprod_{\cat C'}\fnt U_{\CA,\cat C'}(\class{K})\to \coprod_{\cat C}\fnt U_{\CA,\cat C}(\class{K})$ denote the quotient map. Then the diagram in Fig.~\ref{Fig:Chis2} commutes. 
Consequently, 
 on the one hand,  if $\chisub{\class{K}}$ is surjective, then $\chisub{\class{K}}'$ is surjective, which proves (i). On the other hand, if $\chisub{\class{K}}'$ is injective, then $\rho$ is injective, and therefore an isomorphism. This implies that $\chisub{\class{K}}$  is injective, proving (ii).
\end{proof}
\begin{figure} [ht]
\begin{center}
\begin{tikzpicture} 
[auto,
 text depth=0.25ex,
 move up/.style=   {transform canvas={yshift=2.5pt}},
 move down/.style= {transform canvas={yshift=-2pt}},
 move left/.style= {transform canvas={xshift=-2.5pt}},
 move right/.style={transform canvas={xshift=2.5pt}}] 
\matrix[row sep= .9cm, column sep= .9cm]
{ \node (CUK) {$\coprod_{\cat C}\fnt U_{\CA,\cat C}(\class{K})$ };  
& \node (UCK) {$U_{\CA,\cat C}(\coprod_{\CA}\class{K})$};\\
\node (CUK1) {$\coprod_{\cat C'}\fnt U_{\CA,\cat C'}(\class{K})$ };  
& \node (UCK1) {$U_{\CA,\cat C'}(\coprod_{\CA}\class{K})$};\\
 };
\draw [->] (CUK) to node  {$\chisub{\class{K}}$} (UCK);
\draw [->] (CUK1) to node  [swap]  {$\chisub{\class{K}}'$} (UCK1);
\draw [->>] (CUK1) to node 
{$\rho$} (CUK);
\draw [<->] (UCK) to node {${\rm Id}$} (UCK1);
\end{tikzpicture}
\end{center}\caption{The proof of Theorem~\ref{Theo:RetractCons}}\label{Fig:Chis2}
\end{figure}

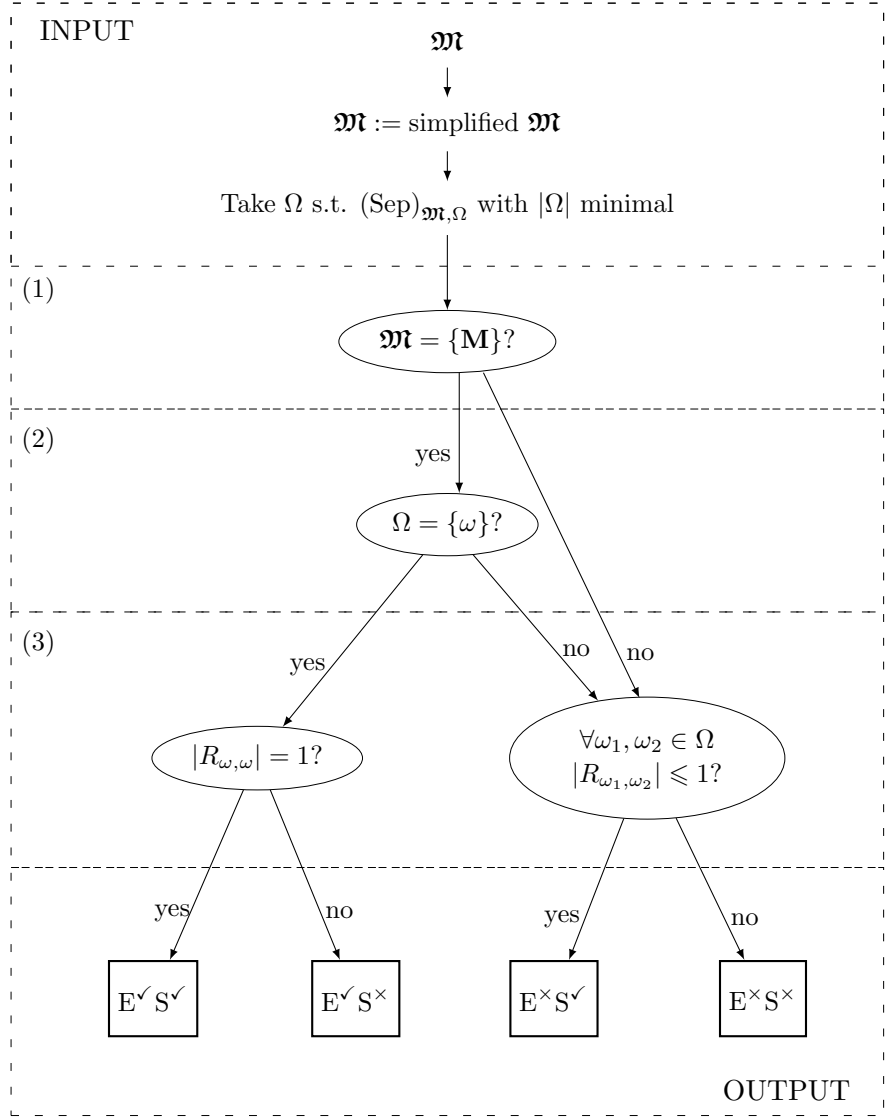
\begin{figure} [ht]
\begin{center}
￼￼\begin{tikzpicture}
[move up/.style=   {transform canvas={yshift=5pt}},
 move down/.style= {transform canvas={yshift=-5pt}},
 move left/.style= {transform canvas={xshift=-5pt}},
move wayright/.style = {transform canvas = {xshift =8pt}},
 move right/.style={transform canvas={xshift=4.5pt}}, 
move righta/.style ={transform canvas={xshift=5.2pt}}, 
 move uup/.style=   {transform canvas={yshift=8pt}},
 move ddown/.style= {transform canvas={yshift=-12pt}},
 move dddown/.style= {transform canvas={yshift=-15pt}}, 
move lleft/.style= {transform canvas={xshift=-13pt}},
%move wayright/.style = {transform canvas = {xshift =8pt}},
 move rright/.style={transform canvas={xshift=13pt}}, 
%move righta/.style ={transform canvas={xshift=5.2pt}},
start chain, node distance=4mm,
conditionbox/.style={
shape=rectangle,
minimum size=10mm,
 thick,
draw=black,
},
infobox/.style={
shape=rectangle,
minimum size=7mm,
},
yesno/.style={
shape=ellipse, 
minimum size=0mm,
very  thin,
draw=black
},
%blank/.style={
%}
]
[move wayright];
\node (CM) [on chain,infobox] {$\CM$};

\node (CMsimp) [on chain,infobox,below= of CM] {$ \CM:= \text{simplified } \CM$};
\node (Sep) [on chain,infobox, below= of CMsimp] 
{
Take  $\Omega$ s.t.
 $ \text{(Sep)}_{\CM,\Omega}$ with $|\Omega| $ minimal
};  
\begin{scope} [node distance=10mm, start branch=godown]
\node (QoneM) [on chain,yesno,below =of Sep] {$ \CM = \{ \M\}$?};
%%% RQs branches off Qone M
\begin{scope} [node distance =50mm, start branch =RQ1]
\node (blank3) [on chain= going below] {};
\begin{scope} [node distance=10mm, start branch= RQleft]
\node (QRa) 
[yesno, on chain=going left] { $ |R_{\w,\w}|=1$?}; 
\end{scope}
\end{scope}

\begin{scope} [node distance=50mm, start branch= RQ2]
\node (blank4) [on chain=going below] {};
%\begin{scope} [node distance =34mm, start branch=RQs]
\begin{scope} [node distance =7mm, start branch =RQright]
\node (QRb) [yesno, on chain=going right] { \begin{tabular}{c}$  
\smash \forall \w_1,\w_2 \in \Omega$\\
$\smash |R_{\w_1,\w_2}|\leq 1$?
\end{tabular}};
\end{scope}
\end{scope}
\node (blank1) [node distance=12mm, below = of QoneM] {};

\begin{scope} [node distance = 16mm, start branch=godownw]
\node (Qoneomega) [yesno,on chain,below=of QoneM] {
$\Omega=\{ \w\}$?
};
\end{scope}
%%%subbranches from QoneM %%%%%%%%%%%%%%%%%
\begin{scope} [node distance=82mm, start branch=forconds12]
\node (blankl) [on chain=going below] {};
\begin{scope} [node distance =5mm, start branch=conds12]
\node (newcond2) [conditionbox, on chain=going left] {\eysn};
\begin{scope} [node distance = 15mm,start branch = forconds12l]
\node (newcond1) [conditionbox, on chain=going left]
{\eysy};
\begin{scope} [node distance = 8mm,start branch =forsecs]
\node (junk) [on chain = going left] {};
\begin{scope} [node distance = 4mm,start branch = forinput]
\node (junk2) [on chain = going right] {};
\end{scope}
\end{scope}
\end{scope}

\end{scope}
\end{scope} 
\begin{scope} [node distance =82mm, start branch=forconds34]
\node (blankr) [on chain=going below] {};
\begin{scope} [node distance =7mm, start branch=conds34]
\node (newcond3) [conditionbox, on chain=going right] {\ensy};
\begin{scope} [node distance = 16mm, start branch = forcondsr]
\node (newcond4) [conditionbox, on chain=going right] {\ensn};
\end{scope}
\end{scope} 
\end{scope}
\end{scope}
%\node (table) [node distance= 5mm, below = of newcond2] {See Table \ref{Tab:ProcedureByType}};
\node (out) [node distance = 4.5mm, below = of newcond4] {\large \hspace*{5mm} OUTPUT};
\node (o4) [node distance = 12mm, above = of junk] {};
\node (q3) [node distance = 43mm, above = of junk] {(3)};
\node (q2) [node distance = 70mm, above = of junk] {(2)};
\node (q1) [node distance = 90mm, above = of junk] {(1)};
\node (i0) [node distance = 125mm, above = of junk2] {\large INPUT};

\draw [ultra thin, loosely dashed] (-5.8,-3) rectangle (5.8,.5);
\draw [ultra thin, loosely dashed] (5.8,.5) rectangle (-5.8,-4.9);
\draw [ultra thin, loosely dashed] (-5.8,-4.9) rectangle (5.8,-7.6);
\draw [ultra thin, loosely dashed] (5.8,-7.6) rectangle (-5.8,-11);
\draw [ultra thin, loosely dashed] (-5.8,-11) rectangle (5.8,-14.3);
%\draw [ultra thin, loosely dashed] %(4.9,-16) -- %rectangle
 %(-4.8,-16) -- (-4.8,-11.3);

%
\draw [-latex] (CM) to (CMsimp);
\draw [-latex] (CMsimp) to (Sep);
\draw [-latex] (Sep) to (QoneM);
\draw [-latex,move right] (QoneM) to  node [xshift=-3.5mm,yshift=-5mm,move up] {yes} (Qoneomega);

\draw [-latex] (QRa) to node [move lleft,yshift=-5mm] {yes} (newcond1);
\draw [-latex] (QRa) to node [move rright,yshift=-5mm] {no} (newcond2);
\draw [-latex] (QRb) to   node [move lleft,yshift=-4mm] {yes} (newcond3);
\draw [-latex] (QRb) to node [move rright,yshift=-4mm] {no}  (newcond4); 
%\node (blank5) [node distance =0mm, below= of blank4 ]{};
%\node (blank6) [node distance=0mm, right =of blank1] {};
\draw [-latex] (Qoneomega) to node [xshift=5.4mm, yshift=-3mm,
swap] {no} (QRb);

\draw [-latex,move wayright] (QoneM) to  node [xshift=10.5mm, yshift=-44.9pt,swap] {no} (QRb);
\draw [-latex] (Qoneomega) to node [yshift=-3.5mm,xshift=-6mm, swap] {yes} (QRa); 

\end{tikzpicture}
\end{center}
\caption{Flowchart for testing conditions (E) and (S)} \label{flowchart}
\end{figure}

\begin{table}[ht]
\begin{center}
\begin{tabular}[t]{|c|p{9.6cm}|}
\hline
&   \\[-.3cm]
 Property  &\hfil Variety \hfill     \\[.15cm]
\hline 
& \\[-.3cm]
 \eysy & Boolean Algebras\newline
De Morgan algebras \newline
$n$-valued pre-\L ukasiewicz--Moisil algebras\newline
$n$-valued pre-Moisil algebras \newline
Stone algebras\\[.15cm]
\hline 
&  \\[-.3cm]
 \eysn & Heyting algebra varieties of the form $\cat{G}_n$  ($n \geq 3$)\newline
Pseudocomplemented lattice varieties $\cat B_n$ ($n\geq 2$)\newline
$Q$-lattice varieties $\cat{D}_{p0}$ and $\cat{D}_{q1}$ ($p\geq 1$ and $q\geq 0$)
\\[.15cm]
\hline 
&  \\[-.3cm]
 \ensy & Kleene algebras\\ & MV-algebra varieties containing no $\textbf{\L}_{p\cdot q}$  ($p,q$  distinct
primes)\newline
$n$-valued \L ukasiewicz--Moisil algebras ($n\geq 2$)\newline
$n$-valued Moisil algebras ($n\geq 2$)\\[.15cm]
\hline 
&  \\[-.3cm] 
\ensn & 
Non-singly generated varieties of Heyting algebras
 \newline 
 MV-algebra  varieties containing some $\textbf{\L }_{pq}$ ($p,q$ 
 distinct primes)
\newline
$Q$-lattice varieties $\cat{D}_{pq}$ ($ q\geq 2$)
 \\[.15cm]
\hline 
\end{tabular}
\medskip
\end{center}\caption{A medley of examples} 
\label{table:classify}
\end{table}

\begin{table}[ht]
\begin{center}
\begin{tabular}[t]{|c|p{9.6cm}|}
\hline
&   \\[-.3cm]
 Property  &\hfil Strategy for  describing coproducts \hfill     \\[.15cm]
\hline 
& \\[-.3cm]
 \eysy &
Simple
piggyback (=\DP-based) duality available. \\[.15cm]
\hline 
&  \\ [-.3cm]
\eysn  &
{ 
No enveloping quasivariety, but there 
is a single-sorted
piggyback duality
 which may lead to a 
description of coproducts} \\[.15cm]
\hline 
&  \\[-.3cm]
\ensy &
{\tt if} 
we have  an   enveloping  quasivariety $\cat B$
with property \eysy 
 \newline
\hspace*{.5cm} {\tt then}    
combine simple piggyback duality (=\DP-based\\
& \hspace*{1.35cm} duality) 
for $\cat B$  with Prop.~\ref{Prop:Reflect}  
\newline
\hspace*{1cm}  {\tt else} 
seek a two-way translation between 
piggyback \newline 
\hspace*{1.15cm} $\phantom{else}$ 
and
\DP-based
 dualities \\[.15cm]
 \hline 
&  \\[-.3cm]
\ensn &
{\tt if}
we have  an enveloping quasivariety $\cat B$ with property \eysn
 \newline 
\hspace*{.5cm} {\tt then}  
combine 
single-sorted piggyback duality for $\cat B$\\
& \hspace*{1.35cm}  with  Proposition~\ref{Prop:Reflect}  
\newline
\hspace*{1cm} {\tt else} 
seek a two-way translation  between 
piggyback\newline
\hspace*{1.15cm} $\phantom{else}$  and~\DP-based  dualities \\[.15cm]
\hline
\end{tabular}
\end{center}
\begin{alignat*}{2}
& \qquad\text{E}^{\text{\checkmark}}\!{/}\,\text{E}^\times: \quad &&
\text{see Theorem~\ref{Theo:OntoCop}}\hspace*{6cm} \\
& \qquad\text{S}^{\text{\checkmark}}\!{/}\,\text{S}^\times: \quad 
&&\text{see Theorem~\ref{Theo:EmbeddingCoprod}}   
\end{alignat*}
\caption{Obtaining descriptions of  coproducts, by
type} 
\label{Tab:ProcedureByType}
\end{table}

\section{Applications}\label{Sec_Ex}
Finitely generated quasivarieties $\CA$ of $\CCD$-based algebras can be classified into  four  types  according to whether the forgetful functor $\fnt U \colon {\CA} \to \CCD$  
 satisfies or fails to satisfy 
 the conditions (S) and (E).  
Theorems~\ref{Theo:EmbeddingCoprod} and~\ref{Theo:OntoCop} suggest a strategy for analysing the properties of $\fnt U$. 
 This is depicted in the flowchart in Fig.~\ref{flowchart}.
The flowchart comes in three parts:  an input section in which 
we assemble the relevant information about the class~$\CA$ to be considered; a series of questions; and an output section in which 
the class $\CA$ is classified according to the answers to the questions.  
At the outset, we assume that, or arrange that,  $\CA$ is expressed in `simplified' form,
that is, that it is presented as ${\CA = \ISP(\CM)}$, where 
${\CM \subseteq \text{Si}(\CA)}$ 
 and $\CM$ has  minimal cardinality  
(see the remarks following Theorem~\ref{MainTheo}).  
 The first question posed---whether ${|\CM|=1}$---could be bypassed.  Still it is useful to know the answer:  
 $|\CM| =1$ is exactly the condition
 for $\CA$ to admit a single-sorted 
natural duality 
(Theorem~\ref{Theo:CoproToNatDual}(i)), and this property is advantageous; moreover,  $|\CM| =1$ is also the condition for 
$\CA$ to admit free products (Theorem~\ref{Theo:FreeProd}).  
In the output section of the 
flowchart,  and in Table~\ref{Tab:ProcedureByType} below,  we adopt an abbreviated notation for results: $\eysn$ indicates that condition (E) holds and condition 
(S) fails, and so on.  

In this section we will present 
 examples of varieties for the four possible combinations of properties. 
We summarise our examples in Table~\ref{table:classify};   
definitions of the listed
classes,
all of which are $\CCD$-based varieties,
  are 
recalled below.
In the table,
we leave it tacit that the quasivarieties 
 are of the form 
$\ISP(\CM)$, where $\CM$ is a finite set of finite algebras.  When we
refer to a quasivariety being 
\defn{singly generated},
we mean that $\CM$ contains a single algebra. 
Our catalogue of examples is by no means exhaustive 
and 
there are many other classes of algebras we could  
equally well have used to illustrate our methods.

For each of our  selected examples, we shall answer the questions in each section of the flowchart, 
giving references for known results and including proofs only when we
 could not find these in the literature.
The output section  of Fig.~\ref{flowchart}  will then tell us the properties of  $\fnt U$
for each variety $\CA$.
Once we have these  properties to hand in a given case,
we want to proceed to describe coproducts in the variety 
concerned.  
When the algebraic structure is completely determined by the lattice reduct, as in the case of Heyting algebras or pseudocomplemented lattices,  the description of the behaviour of $\fnt U$  on coproducts, completely determines  coproducts in the corresponding class. Even when this is not the case, we 
may still be able to describe coproducts.
 Table~\ref{Tab:ProcedureByType} indicates some  general strategies 
(motivated by Theorem~\ref{Theo:CoproToNatDual})
for achieving  this. 

In case coproducts are  preserved,  
 Theorem \ref{Theo:CoproToNatDual}(iv) tells us that we are dealing with a quasivariety for which a simple 
piggyback duality is available and, moreover, this duality can  equivalently
be viewed as a 
\DP-based duality.
When $\fnt U_{\CA}$ 
does not preserve coproducts, the  primary
idea is to  use the results in Section~\ref{Sec:CopSub} to find a quasivariety $\cat{B}$ containing $\CA$ such that 
$\fnt U_{\cat{B}}$ has better properties than $\fnt U_{\CA}$ (here we annotate the functor to indicate the class to which it refers).
Such 
 an \defn{enveloping quasivariety for $\CA$} is not in general unique
and 
which  choice  we make will 
depend on how much information 
we have about coproducts in the possible quasivarieties $\cat B$. 
Depending on the properties of coproducts of this enveloping quasivariety $\cat B$ that the class $\CA$ does not have, we have
simpler natural dualities for $\cat B$ than  for $\CA$; 
see Theorem~\ref{Theo:CoproToNatDual}.
 Again we will 
cite earlier literature where appropriate.

We now discuss in turn the varieties listed in Table~\ref{table:classify}.  For complete clarity we shall write $\fnt U_{\CA}$, for each given
choice of $\CA$, rather than~$\fnt U$.

\subsection*{De Morgan Algebras, $\cat{DM}$}\

An algebra $\A=(A,\wedge,\vee,\neg,0,1)$ 
is a \defn{ De Morgan algebra }
 if its lattice reduct $(A,\wedge,\vee,0,1)$ is in $\CCD$  and  $\neg$ is a unary operation satisfying the equations $x\approx\neg\neg x$ and $\neg(x \wedge y)\approx\neg x \vee \neg y$ and $\neg 0 \approx 1$. 
Then $\cat{DM} = \ISP(\mathbf{4})$, where 
$\mathbf{4}=(\{0,a,b,1\},\wedge,\vee,\neg,0,1)$ denotes the four-element De Morgan  algebra with two $\neg$-fixpoints.
Let $\eta$ be the automorphism of $\mathbf{4}$ which interchanges
$a$ and~$b$ and let $\w\colon \fnt U_{\cat{DM}}(\mathbf{4})\to \two$ be defined by  $\w(a)=\w(1)=1$ and $\w(b)=\w(0)=0$. Then $ \text{(Sep)}_{\mathbf{4},\w}$ is satisfied. The answers to (1) and (2) in Fig.~\ref{flowchart} are yes. Note that Lemma~\ref{unique-max}  is applicable here, so the answer to (3) is also yes. Therefore, the functor $\fnt U_{\cat{DM}}$ preserves coproducts. 

By Theorem~\ref{genpig}, the structure   
$\MT = \langle \,\{0,a,b,1\}, r, e, \Tp\, \rangle $,  
where
$$r=\{\,(0,0),(0,a),(a,a),(b,0),(b,a),(b,b),(b,1),(1,a),(1,1)\,\}, 
$$
yields a natural duality for~$\cat{DM}$ and coproducts in~$\cat{DM}$ correspond to cartesian products in the dual category. 
The duality and the description of coproducts
stemming from it 
 were already developed in~\cite{CF77}.  There the duality was introduced as a  \DP-based 
 one.  In \cite[Section~3.15]{CD98} this same duality
is presented from the natural duality perspective, in accord with 
Corollary~\ref{Cor-to-RevEng}. 

\subsection*{Kleene algebras, $\cat{K}$ }\

An algebra $\A$ is  
a \defn{Kleene algebra} if it is a 
De~Morgan algebra satisfying the Kleene condition 
$x \wedge \neg x \leq y \vee \neg y$. We have  $\cat{K}=\ISP(\mathbf{3})$
where~$\mathbf{3}$ denotes the subalgebra $(\{0,a,1\},\wedge,\vee,\neg,0,1)$ of $\mathbf{4}$ 
determined by $\{0,a,1\}$. 
Then the answer to (1) is yes. 
The only endomorphism of $\mathbf{3}$ is the identity. 
Thus the only~$\Omega$ that satisfies $\text{(Sep)}_{\mathbf{3},\Omega}$ is $\Omega=\fnt{HU}_{\cat{K}}(\mathbf{3})=\{\w_1,\w_2\}$, where $\w_1,\w_2\colon\fnt U_{\cat{K}}(\mathbf{3})\to\two$ are defined by $\w_1(a)=\w_1(1)=1$, $\w_1(0)=0$ and $\w_2(1)=1$, $\w_2(a)=\w_2(0)=0$. Therefore the answer to (2) is no. 
By Lemma~\ref{unique-max}, for each pair $\w_i,\w_j$
($i,j = 1,2$), we have $|R_{\w_i,\w_j}|=1$, 
so 
the answer to (3) is yes. 
Hence 
Kleene algebras admit free products and $\fnt U_{\cat{K}}$ has the property (S) but not property~(E).

Guided by  Table~\ref{Tab:ProcedureByType}, to  describe coproducts in $\cat{K}$  fully we seek 
 an enveloping variety
 $\cat{B}$ such that $\fnt U_{\cat{B}}$ preserves coproducts and then  determine the retraction~$\fnt R_{\cat{K}}$. 
This is exactly the procedure followed  in  \cite{Ci79} and
in \cite{CF79} to describe coproducts in $\cat{K}$ using $\cat{DM}$ as enveloping variety.
An alternative proof can be obtained using Lemma~\ref{jointsurj},  Theorem~\ref{Theo:EmbeddingCoprod} and 
Corollary~\ref{Cor:Range}. 
Theorem~\ref{genpig}  can then be used to prove that  
$
(\,\{0,a,1\}\overset{.}{\cup} \{0,a,1\};\{R_{\w_i,\w_j}\mid i,j \in\{ 1,2\}\},\Tp\, )
$ 
yields a multisorted natural duality on $\cat{K}$.
(See  \cite{DP87}, where this piggyback duality for Kleene algebras was first developed.)

\subsection* {Pseudocomplemented distributive lattices}\

We consider the countable chain of non-trivial finitely generated subvarieties ~$\cat {B}_n$ of the variety 
 $\cat{B}_\omega$  
of pseudocomplemented distributive lattices.  An 
algebra  $\A=(A,\wedge,\vee,^*,0,1)$  is in $\cat {B}_{\w}$ if
 $(A,\wedge,\vee,0,1)\in \CCD$  and  $^*$ is a unary operation satisfying
 $x\wedge y\approx 0\mbox{ if and only if } x\leq y^*$.
For $0 \leq n < \omega$ the variety $\cat B_n $ is expressible as  $\ISP(\str{B}_n)$
where $\str{B}_{n}$  has as underlying lattice the Boolean lattice with $n$ atoms with a new top element  adjoined; here
 $\cat B_0$ and $\cat B_1$ correspond  to the varieties of Boolean algebras and of Stone algebras,
respectively. 
(See for example \cite{BaDw} for details.)

Piggyback dualities dualities for the classes $\cat B_n$ were studied in \cite{DP93}, building on earlier work  in \cite{DWpig}.
We already noted that the answer to our question (1) is yes for each $\cat B_n$. 
In  \cite[p.~48]{DP93} 
it is observed that for each $n$ there exists  $\w_n$ such that $\text{(Sep)}_{\B_n,\w_n}$ holds,  and so the answer to (2) is yes for each $n\geq 0$.  
In \cite[Theorem~3.6]{DP93}, it is proved that $|R_{\w_n,\w_n}|$ is equal to the number of partitions of~$n$. 
Thus the answer to (3) is yes if and only if  $n\leq 1$. 
In summary,  $\fnt U_{\cat{B}_{n}}$, for $n \leq 1$,  preserves coproducts and  $\fnt U_{\cat{B}_{n}}$,   for $n\geq 2$,  satisfies condition (E) but not condition (S).

\subsection*{Quasivarieties of Heyting algebras generated by finite chains}\

We recall that
an algebra $\A=(A,\wedge,\vee,\to,0,1)$ is a Heyting algebra if 
$(A,\wedge,\vee,0,1) \in \CCD$ 
and $\to$ is a binary operation satisfying 
$x\wedge y\leq z$ if and only if $x\leq y\to z$.
The implication~$\to$ is uniquely determined by the lattice structure. 
Therefore determining the properties of~$\fnt U$ with respect to coproducts suffices for a full understanding of  coproducts in quasivarieties of Heyting algebras.
 
For each $n\in\{2,3,\ldots\}$, let 
$
\str{C}_n
$ 
denote the $n$-element Heyting chain. 
and let $\cat{G}_n=\ISP(\str{C}_n)$.
The algebras in $\cat{G}_n$ are known as \defn{$n$-valued G\"odel algebras}, 
since they form the algebraic counterpart of the $n$-valued G\"odel logic (see \cite{G86,H98}). It is easy to see that $\cat{G}_2$ is term-equivalent to the variety of Boolean algebras. So we 
 restrict to 
the case $n\geq 3$. 
For each $n$ the condition
$\text{(Sep)}_{\str{C}_n,\w_n}$ is satisfied if we take 
$\w_n\colon\fnt{U}_{\cat{G}_n}(\str{C}_n)\to\two$ to be the map 
determined by   
$\w_n^{-1}(1) = \{ 1\}$.  
 Therefore in each $\cat{G}_n$ the answer to (1) and (2)   is yes.
The algebra $\str{C}_3$ belongs to $\cat{G}_n$ and it can  easily  be checked that
 $|R_{w_3,w_3}|=\{r_1,r_2\}$, where  $r_1=\{(0,0),(d,d),(1,1)\}$ and $r_2=\{(0,0),(d,1),(1,1)\}$. 
Then Theorem~\ref{Theo:EmbeddingCoprod}(4) 
proves  that the answer to (3) is no for each $\cat{G}_n$.  
Thus,  for each $n\geq 3$, the functor $\fnt U_{\cat{G}_n}$ satisfies (E) but not (S).  

Coproducts can be described  
completely by using
the simple piggyback dualities  for  the classes $\cat{G}_n$ derived in \cite{DWpig} (or see \cite[Section~7.3]{CD98}).
 In \cite{CPXX} we undertake an in-depth study of coproducts in the classes
$\cat G_n$. 
A different approach to coproducts of finite
G\"odel algebras $\ope{V}(\{\str{C}_n\mid n\geq 1\})$ 
has been developed in \cite{DM}. 
The results in \cite{DM} combined with Proposition~\ref{Prop:Reflect}  
provide another description of coproducts for finite algebras in $\cat{G}_n$ for each~$n\geq 2$.

\subsection*{MV-algebras}\

An algebra $(A,\oplus,\neg,0)$ is an \defn{MV-algebra} if $(A,\oplus,0)$ is a 
commutative mon\-oid satisfying 
$\neg(\neg x)\approx x$,  $\neg 0\oplus x\approx\neg 0$ and
 $\neg(\neg x\oplus y)\oplus y\approx\neg(\neg y\oplus x)\oplus x$.
The variety of MV-algebras is the algebraic counterpart of \L ukasiewicz infinite-valued logic (see \cite{CDM00}).
The terms
$\neg(\neg x\oplus y)\oplus y$, $\neg(x\oplus \neg y)\oplus \neg y$, $0$ and 
$\neg 0$
 determine a bounded distributive lattice structure on any  MV-algebra $\A$.
For each finitely generated variety $\CA$ of MV-algebras there is a finite set~$\CM$ of finite MV-chains such that $\CA=\ISP(\CM)$ 
\cite[Chapter~8]{CDM00}. 
The $n$-element MV-chain is denoted by~$\textbf{\L}_{n-1}$. (We deviate here from the notation in \cite[Section 3.5]{CDM00} where the  $n$-element MV-chain is denoted by~$\textbf{\L}_{n}$. 
We do this 
 so  that $\textbf{\L}_{m}\in\ope{IS}(\textbf{\L}_{n})$ if and only if $m$ divides $n$.)
The variety $\ISP(\textbf{\L}_1)$ is term-equivalent to the variety of Boolean algebras \cite[Corollary~8.2.4]{CDM00}.  Therefore~$\fnt U_{\ISP(\textbf{\L}_1)}$ preserves coproducts. 
Now assume that $n \geq 2$.
Since the only endomorphism of
~$\textbf{\L}_n$ is the identity,
$\text{(Sep)}_{\textbf{\L}_n,\Omega}$ can be satisfied only by taking  $\Omega=\fnt{HU}_{\ISP(\textbf{\L}_n)}(\textbf{\L}_n)$. 
Then 
the answer to (1) is yes but to (2) is no. 
If $n$ is a power of a prime,  
it can be proved that $|R_{\w_1,\w_2}|=1$ for each $\w_1,\w_2\in\fnt{HU}_{\ISP(\textbf{\L}_n)}$ and in these cases the answer to (3) is yes and 
$\fnt{U}_{\ISP(\textbf{\L}_n)}$ satisfies (S) but not (E). 
In case~$n$ has at least two distinct prime divisors, 
 it can be checked that  there exist 
$\w_1,\w_2\in\fnt{HU}_{\ISP(\textbf{\L}_n)}$ such that $|R_{\w_1,\w_2}|>1$ and then 
the answer to (3) is no. Thus $\fnt{U}_{\ISP(\textbf{\L}_n)}$ 
satisfies neither  (E) nor (S).
If a quasivariety of MV-algebras is generated by a finite family of MV-chains $\{\textbf{\L}_{n_1},\ldots, \textbf{\L}_{n_m}\}$ 
with $m>1$,
then the answer to (1) and (2) is no,  and the answer to (3) is yes if and only if $n_i$ is a prime power
 for 
 $1\leq i\leq m$, that is, if no $\textbf{\L}_{p\cdot q}$ with $p$ and $q$ distinct primes 
 belongs to $\ISP(\textbf{\L}_{n_1},\ldots, \textbf{\L}_{n_m})$.

To  describe coproducts for singly-generated quasivarieties of 
MV-algebras we have available a special kind of natural duality since each $\textbf{\L}_n$ is a discriminator algebra (see \cite{Ni2001} and \cite[Section 3.12]{CD98}). 
If a quasivariety of MV-algebras is generated by $\{\textbf{\L}_{n_1},\ldots, \textbf{\L}_{n_m}\}$ then each $\textbf{\L}_{n_i}\in\ope{S}(\textbf{\L}_{k})$,
 where $k=n_1
\dots
 n_m$. 
We can then combine the duality for  $\ISP(\textbf{\L}_{k})$ and Proposition~\ref{Prop:Reflect} (see \cite{MNT2007}, where this approach was
 used to develop a duality for finitely generated quasivarieties of MV-algebras).
 Coproducts of MV-algebras have also been studied, using different tools, in  \cite{Pa} and \cite[Chapter~7]{Mu11}.

\subsection*{$Q$-lattices}\

An algebra $\A=(A,\wedge,\vee,\nabla,0,1)$ 
is a  \defn{distributive lattice with a quantifier} (\defn{$Q$-lattice} for short) 
 if $(A,\wedge,\vee,0,1)\in \CCD$  and  $\nabla$ is a unary operation satisfying the conditions:
$$\nabla 0\approx 0,\ \ x\leq\nabla x, \ \ \nabla(x\vee y)\approx\nabla x\vee\nabla y\ \mbox{ and }\ \nabla(x\wedge\nabla y)\approx \nabla x\wedge\nabla y.$$
The variety 
of $Q$-lattices was introduced
in \cite{Ci91}. There 
 it is proved 
that the lattice of subvarieties of $Q$-lattices  is a 
chain. 
Each proper non-trivial subvariety is generated by a finite algebra $\str{D}_{pq}$ where $p,q\in\{0,1,\ldots\}$, with $p=q=0$  excluded  (we refer to \cite[Section~4]{Ci91} for precise definitions). 

In \cite{Pr93}, the second author presented piggyback natural dualities for these varieties $\cat{D}_{pq}=\ope{V}(\str{D}_{pq})$. By \cite[Theorem~3.6]{Pr93}, 
 the answer to (1) and to (2) is yes if and only if $q\leq 1$ and the answer to both questions  is no otherwise. From \cite[Theorem~3.10]{Pr93}, it follows that the answer to (3) is yes only 
for 
 $\cat{D}_{10}$ and $\cat{D}_{01}$ (see also \cite[Table~1]{Pr93}). 
We deduce  
that $\fnt U_{\cat{D}_{10}}$ and $\fnt U_{\cat{D}_{01}}$ preserve coproducts, that $\fnt U_{\cat{D}_{p0}}$ and $\fnt U_{\cat{D}_{q1}}$,
for  $p\geq 2$ and $q\geq 1$, satisfy condition (E) but not condition (S),
 and that $\fnt U_{\cat{D}_{pq}}$ with $q\geq 2$
satisfies neither (E) nor (S). 
(To see 
that 
$\fnt U_{\cat{D}_{10}}$ and $\fnt U_{\cat{D}_{01}}$ 
preserve coproducts
it suffices to observe that 
$\cat{D}_{10}$ and $\cat{D}_{01}$ are term-equivalent to   $\CCD$ and  to Stone algebras, respectively.)

Coproducts of $Q$-lattices have not been studied in the literature, apart from what is included  in this subsection, and the observation made in \cite{Ci96} that coproducts of $Q$-lattices do not correspond to cartesian products in 
the \DP-based 
so coproducts are not preserved.
The natural duality developed in \cite{Pr93} can be used to describe coproducts on $\cat{D}_{pq}$, but the 
the proliferation of relations in the alter ego as $p$ and $q$ increase 
makes a naive application of this technique unfeasible when $p+q$ is large.

\subsection*{$n$-valued pre-Moisil and pre-\L ukasiewicz--Moisil algebras}\

Our terminology here largely follows that in 
\cite{BFGR}, which provides our primary source for facts about 
the classes of algebras we now consider.  
An algebra
\[
\A=(A,\wedge, \vee,0,1,\{ f_i,\overline{f}_i\mid i\in\{1,\ldots,n-1\})
\]  
is called an  \defn{$n$-valued pre-\L ukasiewicz--Moisil  algebra} (or \defn{pre-$\text{\rm LM$n$}$-algebra} for short) if it satisfies the following:
\begin{enumerate}
\item $(A,\wedge, \vee, 0,1)\in \CCD$; 
\item each $f_i$ is a $\CCD$-endomorphism;
\item $f_i(a)\wedge \overline{f}_i(a)=0$ and $f_i(a)\vee \overline{f}_i(a)=1$, for each $i\in\{1,\ldots,n-1\}$ and each $a\in A$;
\item $f_i\circ f_j=f_j$, for each $i,j\in \{1,\ldots,n-1\}$;
\item if $i\leq j$, then $f_i(a)\leq f_j(a)$ for each $a\in A$.
\end{enumerate}
 An algebra $(A,\wedge, \vee,\neg,0,1, \{ f_i\mid i\in\{1,\ldots,n\})$
  is an \defn{$n$-valued pre-Moisil algebra} (pre-$\text{\rm M}n$-algebra for short) if its reduct $(A,\wedge,\vee,\neg,0,1)$ 
is a De Morgan algebra, and  the algebra 
$(A,\wedge, \vee,0,1, \{ f_i,\neg\circ f_i\mid i\in\{1,\ldots,n-1\})$,
is a pre-$\text{\rm LM}n$-algebra.
Following 
\cite{BFGR}, we let $\LMn^0$ and  $\Mn^0$ 
  denote the classes of pre-$\text{\rm LM}n$-algebras and 
pre-$\text{\rm M}n$-algebras, respectively.

In \cite[Chapter 6]{BFGR},  
\DP-based dualities for  $\LMn^0$ and $\Mn^0$ were set up. 
In \cite[Lemmas 5.20 and 5.24]{BFGR},
 it is proved that this duality sends coproducts into cartesian products. From this we
 infer that $\fnt{U}_{\LMn^0}$ and $\fnt{U}_{\Mn^0}$ preserve coproducts. 
If it happens that $\LMn^0$ is 
 finitely generated then Theorem~\ref{MainTheo} implies that there exist an algebra $\mathbf{L}^0_n\in\LMn^0$ and a
 lattice homomorphism $\w\colon\fnt{U}_{\LMn^0}(\mathbf{L}^0_n)\to \two$, such that $\ISP(\mathbf{L}^0_n)=\LMn^0$ and  $\text{(Sep)}_{\mathbf{L}^0_n,\w}$. Similarly, if~$\Mn^0$ is 
 finitely generated, then there exist $\mathbf{M}^0_n\in\Mn^0$ and a
 lattice homomorphism $\w\colon\fnt{U}_{\Mn^0}(\mathbf{M}^0_n)\to \two$, such that $\ISP(\mathbf{M}^0_n)=\Mn^0$ and  $\text{(Sep)}_{\mathbf{M}^0_n,\w}$ holds. 
Such algebras are not, however, exhibited 
in the literature on 
 $n$-valued Moisil and \L ukasiewicz--Moisil algebras.  We shall remedy this omission.

The algebra
$$
\mathbf{L}^0_n=(\{0,1\}\times\{0,1,\ldots,n-1\}, \wedge,\vee, 0,n-1,\{e_i,\overline{e}_i\mid i\in\{1,\ldots,n-1\}\,)
$$
whose lattice order is the product order of $\{0,1\}\times\{0,1,\ldots,n-1\}$ and its operations are defined as follows:
$$e_i(j,k)=\begin{cases}
 (0,0) & \mbox{if  }k< n-i,\\
 (1,1) & \mbox{otherwise};
\end{cases}
\ \ \ 
\overline{e}_i(j,k)=\begin{cases}
 (1,1) & \mbox{if  }k< n-i,\\
 (0,0) & \mbox{otherwise},
\end{cases}
$$
is a pre-$\text{LM}n$-algebra.
Similarly, we obtain a pre-$\text{Mn}$-algebra 
$$
\mathbf{M}^0_n=(\{0,a,b,1\}\times\{0,1,\ldots,n-1\}, \wedge,\vee,\neg,0,n-1,\{f_i
\mid i\in\{1,\ldots,n-1\}\,)
$$
by equipping $ \{0,a,b,1\}\times\{0,1,\ldots,n-1\}$
with the product lattice order and operations 
$$
f_i(j,k)=\begin{cases}
 (0,0) & \mbox{if  }k<n- i,\\
 (1,1) & \mbox{otherwise},
\end{cases}
$$
and $ \neg(j,k)=(\neg j 
,n-1-k)$,
where $\neg$ on the right-hand side has its interpretation 
in the De Morgan algebra~$\mathbf{4}$. 

We will now prove that the algebras $\mathbf{L}^0_n$ and $\mathbf{M}^0_n$ satisfy 
conditions (B)(i)--(ii)  of Theorem~\ref{MainTheo}. Observe that condition (B)(iii) follows from Lemma~\ref{unique-max}.
\begin{theorem}
For each $n$, the following statements hold:
\begin{newlist}
\item[{\rm(i)}]$\ISP(\mathbf{L}^0_n)=\LMn^0$;
\item[{\rm(ii)}] if $\w_{\mathbf{L}^0_n}\colon \mathbf{L}^0_n\to\two$ is 
 given by $\w_{\mathbf{L}^0_n}(x,y)=x$, then ${\rm (Sep)}_{\mathbf{L}^0_n,\w_{\mathbf{L}^0_n}}$ holds;
\item[{\rm(iii)}]$\ISP(\mathbf{M}^0_n)=\Mn^0$;
\item[{\rm(iv)}] if $\w_{\mathbf{M}^0_n}\colon \mathbf{L}^0_n\to\two$ is given  by $$
\w_{\mathbf{L}^0_n}(x,y)=\begin{cases}
1&\mbox{if }x\in\{a,1\},\\  
0&\mbox{otherwise},
\end{cases}$$ then ${\rm (Sep)}_{\mathbf{M}^0_n,\w_{\mathbf{M}^0_n}}$ holds.
\end{newlist}
\end{theorem}
\begin{proof}
To prove (i) let $\A=(A,\wedge, \vee,0,1,\{f_i,\overline{f}_i\mid i\in\{1,\ldots,n-1\})\in\LMn^0$ and let 
$a,b\in A$ be such that $a\neq b$. There is a lattice homomorphism $h\colon \fnt U_{\LMn^0}(\A)\to \two$ such that $h(a)\neq h(b)$. Let $h'\colon  A\to\{1,\ldots,n-1\}$ be defined by
$$
h'(a)=
\begin{cases}
0&\mbox{if $h(f_i(a))=0$ for } 1 \leq i \leq n-1,\\
\max\{n-i\mid h(f_i(a))=1\}&\mbox{otherwise}.
\end{cases}
$$
It follows that the map $\overline{h}\colon\A\to\mathbf{K}^0_n$ defined by ${\overline{h}(a)=(h(a),h'(a))}$ is a homomorphism of pre-LM$n$-algebras. Since $\overline{h}(a)\neq\overline{h}(b)$ we have that $\A\in\ISP(\mathbf{L}^0_n)$. This concludes the proof of (i).

The proof of (iii) follows by a similar argument but using 
the fact that 
if $\A\in\Mn^0$ and $a,b\in\A$ are such that $a\neq b$, there exists a homomorphism $h$ of De Morgan algebras from $\A$ to $\mathbf{4}$ such that $h(a)\neq h(b)$.

The proof of (ii) follows from the observation that for each $i\in\{1,\ldots,n-1\}$ the maps  $\eta_i\colon \mathbf{L}^0_n\to \mathbf{L}^0_n$ defined by
$$
\eta_i(j,k)=\begin{cases}
(0,k)&\mbox{if }k<i ,\\
(1,k)&\mbox{otherwise}
\end{cases}
$$
are endomorphisms of $\mathbf{L}^0_n$. 
If 
$(j,k)\neq(j',k')$,  then $j\neq j'$ or $k\neq k'$. If $j\neq j'$, then $\w_{\mathbf{L}^0_n}(j,k)\neq\w_{\mathbf{L}^0_n}(j',k')$. If  $k\neq k'$, we may assume without loss of generality that $k< k'$. Then $\w_{\mathbf{L}^0_n}(\eta_{k'}(j,k))\neq\w_{\mathbf{L}^0_n}(\eta_{k'}(j',k'))$.

The proof of (iv) follows similar lines. 
\end{proof}
 Here we have 
an example in which  we use the flowchart in Fig.~\ref{flowchart} in the reverse direction from hitherto.  That is, knowing that the forgetful functor into bounded distributive lattices preserves coproducts, the 
preservation
theorem ensures the existence of a finite algebra that generates the quasivariety and has special properties.

\subsection*{Moisil and \L ukasiewicz--Moisil algebras}\

An algebra $\A=(A,\wedge, \vee,0,1,\{ f_i,\overline{f}_i\mid i\in\{1,\ldots,n\})$  
is called an  \defn{ $n$-valued \L ukasiewicz--Moisil  algebra} if it satisfies the following:
\begin{enumerate}
\item $\A$  is an $n$-valued pre-\L ukasiewicz--Moisil  algebra;
\item if  $f_i(a)= f_i(b)$ for each $i\in\{1,\dots,n\}$, then $a=b$, for each $a,b\in A$.
\end{enumerate}
Similarly, an algebra $\A=(A,\wedge, \vee,\neg,0,1,\{ f_i\mid i\in\{1,\ldots,n\})$  
is called an  \defn{ $n$-valued Moisil  algebra} if it is an pre-M$n$-algebra and satisfies (2).

Let $\LMn$ and $\Mn$ denote the classes of  $n$-valued \L ukasie\-wicz--Moisil and Moisil algebras, respectively.
Let
\begin{align*}
\mathbf{L}_n&=(\{0,1,\ldots,n-1\}, \min,\max,0,n, \{d_i,\overline{d}_i\mid i\in\{1,\ldots,n-1\}\}\,)
\intertext{and }
\mathbf{M}_n&=(\{0,1,\ldots,n-1\}, \min,\max,\neg,0,n,\{d_i
\mid i\in\{1,\ldots,n-1\}\}\,)
\end{align*}
be the  algebras, in $\LMn$ and $\Mn$ respectively,  in which  
the operations are defined as follows:
$$d_i(j)=\begin{cases}
 0 & \mbox{if  }j<n-i,\\
 n-1 & \mbox{otherwise};
\end{cases}
\ \ \ 
\overline{d}_i(j)=\begin{cases}
 n-1 & \mbox{if  }j< n-i,\\
 0 & \mbox{otherwise}; 
\end{cases}$$
 and 
$\neg(j)=n-1-j$.
Then $\LMn=\ISP(\mathbf{L}_n)$ and $\Mn=\ISP(\mathbf{M}_n)$
\cite[Corollary~6.1.9]{BFGR}.
The answer to question  (1)  is yes. The only endomorphisms of $\mathbf{L}_n$ and $\mathbf{M}_n$ are the corresponding identity maps \cite[Theorem~6.1.6]{BFGR}. Thus
to satisfy 
$\text{(Sep)}_{\mathbf{L}_n,\Omega}$
we must take
$\Omega=\fnt{HU}_{\LMn}(\mathbf{L}_n)$. Similarly 
$\text{(Sep)}_{\mathbf{M}_n,\Omega}$ holds only if 
$\Omega=\fnt{HU}_{\Mn}(\mathbf{M}_n)$. 
Then the answer to (2) is yes if and only if $n\leq 2$.
Lemma~\ref{unique-max}  is applicable, so the answer to (3) is yes for each $n\geq 1$.

To  describe coproducts in $\LMn$ and $\Mn$ fully we can use the piggyback duality 
developed by the second author in \cite{Pr95}. 
Alternatively, 
we can apply the strategy of Table~\ref{Tab:ProcedureByType} with $\LMn^0$ and $\Mn^0$ as 
the enveloping varieties. 
 It can be proved that the dual counterparts of the retraction functors
$\fnt R_{\LMn}$ and $\fnt R_{\Mn}$ under the Priestley dualities for $\LMn^0$ and $\Mn^0$ correspond to the assignment to 
$(X,\psi_i,\Tp)$ of  
the closed subspace  
$$\textstyle
Y=\{x\in X\mid \psi_i(x)=x \mbox{ for some }i\in\{1,\ldots,n\}\,\}.
$$
Our conclusions here strengthen 
the results detailed in \cite[Theorem~7.5.9]{BFGR} where duals of coproducts of only two algebras were calculated.

Finally, we remark
 that a translation between the natural
dualities  
for $\LMn$ and $\Mn$ and the Priestley-style dualities for these classes  was provided in \cite[Theorem 3.9]{Pr95}. The process that is described 
there for obtaining 
$\fntH\fnt U_{\LMn}$  and~$\fntH\fnt U_{\Mn}$ from  the corresponding natural duals is exactly the one described in 
Theorem~\ref{Theo:RevEng} applied to
({\L}ukasiewicz--)Moisil algebras.

\end{document}